\documentclass[11pt,a4paper]{amsart}		

\usepackage[applemac]{inputenc}					
\usepackage{xcolor,ulem,textcomp}
\usepackage{amsmath,amsfonts,amssymb,bbm}
\usepackage{lmodern}						
\usepackage[english]{babel}					
\usepackage[weather]{ifsym}	
\usepackage{mathrsfs}				
\usepackage{cancel}
\usepackage{hyperref,enumitem,soul}

\hypersetup{
	colorlinks,
	linkcolor={red!50!black},
	citecolor={blue!50!black},
	urlcolor={blue!80!black}
}

\usepackage{tikz}

\usepackage{amsmath,amsfonts,amssymb,amsthm}

\usepackage[left=3cm,right=3cm,top=3cm,bottom=3cm]{geometry}

\usepackage{mathtools}						

\theoremstyle{plain}
\newtheorem{proposition}{Proposition}[section]
\newtheorem{theorem}[proposition]{Theorem}

\newtheorem{corollary}[proposition]{Corollary}
\newtheorem{lemma}[proposition]{Lemma}

\newtheorem{definition}[proposition]{Definition}

\newtheorem{example}{Example}

\numberwithin{equation}{section}

\newcommand{\R}{\mathbb{R}}					
\newcommand{\C}{\mathbb{C}}					
\newcommand{\N}{\mathbb{N}}					

\newcommand{\eps}{{\varepsilon}}

\def\refer#1{~\ref{#1}}
\def\refeq#1{~(\ref{#1})}
\def\ccite#1{~\cite{#1}}

\def\suite#1#2#3{(#1_{#2})_{#2\in {#3}}}

\def\inte#1{
\displaystyle\mathop{#1\kern0pt}^\circ }

\let\e=\varepsilon

\let\lam=\lambda

\let\vf=\varphi

\let\D=\Delta

\let\S=\Sigma

\let\wh=\widehat

\def\cB{{\mathcal B}}

\def\cD{{\mathcal D}}

\def\cF{{\mathcal F}}

\def\cH{{\mathcal H}}

\def\cL{{\mathcal L}}

\def\cP{{\mathcal P}}

\def\cS{{\mathcal S}}

\def\cW{{\mathcal W}}
\def\cX{{\mathcal X}}

\def\S{{\mathop{\mathbb  S\kern 0pt}\nolimits}}

\def\virgp{\raise 2pt\hbox{,}}
\def\cdotpv{\raise 2pt\hbox{;}}

\def\C{\mathop{\mathbb C\kern 0pt}\nolimits}
\def\EE{\mathop{{\mathbb E \kern 0pt}}\nolimits}
\def\K{\mathop{\mathbb K\kern 0pt}\nolimits}
\def\Q{\mathop{\mathbb Q\kern 0pt}\nolimits}
\def\R{{\mathop{\mathbb R\kern 0pt}\nolimits}}
\def\SS{\mathop{\mathbb S\kern 0pt}\nolimits}
\def\ZZ{\mathop{\mathbb Z\kern 0pt}\nolimits}
\def\T{\mathop{\mathbb T\kern 0pt}\nolimits}
\def\P{\mathop{\mathbb P\kern 0pt}\nolimits}
\def \H{{\mathop {\mathbb H\kern 0pt}\nolimits}}
\newcommand{\ds}{\displaystyle}

\newcommand{\Z}{{\ZZ}}

\newcommand{\beq}{\begin{equation}}
\newcommand{\eeq}{\end{equation}}
\newcommand{\ben}{\begin{eqnarray}}
\newcommand{\een}{\end{eqnarray}}
\newcommand{\beno}{\begin{eqnarray*}}
\newcommand{\eeno}{\end{eqnarray*}}
\newcommand{\bqs}{\begin{equation*}}
\newcommand{\eqs}{\end{equation*}}
\newcommand{\andf}{\quad\hbox{and}\quad}
\newcommand{\with}{\quad\hbox{with}\quad}

\def\equivH#1 {\buildrel\hbox{\tiny {$#1$}}\over \equiv}
\def\simH#1 {\buildrel\hbox{\footnotesize {$#1$}}\over \sim}

\def \bone {\mathbbm{1}}

\def \bC {\mathbb C}

\def \bH {\mathbb H}

\def \bN {\mathbb N}

\def \bR {\mathbb R}
\def \bS {\mathbb S}
\def \bT {\mathbb T}

\def \bX {\mathbb X}

\def \bZ {\mathbb Z}

\def \cB {\mathcal B}

\def \cD {\mathcal D}

\def \cF {\mathcal F}

\def \cH {\mathcal H}

\def \cL {\mathcal L}

\def \cP {\mathcal P}

\def \cS {\mathcal S}

\def \cW {\mathcal W}
\def \cX {\mathcal X}

\def \fg {\mathfrak g}

\def \sL{\mathscr L}

\def \eps {\varepsilon}
\def \vol {{\rm vol}}

\title[Restriction problem]
{Some geometric and spectral aspects\\ of restriction problems
}    
\date{\today}
\author[H. Bahouri \and  V. Fischer]{Hajer Bahouri \and V\'eronique Fischer}
\address[H. Bahouri]
{CNRS  \&  Sorbonne Universit\'e  \\
 Laboratoire Jacques-Louis Lions (LJLL) UMR  7598 \\
4, Place Jussieu\\
75005 Paris, France.}
\email{hajer.bahouri@sorbonne-universite.fr}
\address[V. Fischer]
{Department of Mathematical Sciences, University of Bath, Bath, BA2 7AY, UK}
\email{v.c.m.fischer@bath.ac.uk }

\begin{document}
\setstcolor{red}

\maketitle

\vspace{-1em}
\begin{center}
    \textit{In memory of Ha\"im Brezis}
\end{center}

\setcounter{tocdepth}{2}

\begin{abstract}
This texts commemorates the memory of Ha\"im Brezis and  explores some aspects of the restriction problem, particularly its connections to spectral and geometric analysis. Our choice of subject is motivated by Brezis' significant contributions to various domains related to this problem, including harmonic analysis, partial differential equations, spectral theory, representation theory, number theory, and many others.
\end{abstract}

\makeatletter
\renewcommand\l@subsection{\@tocline{2}{0pt}{3pc}{5pc}{}}
\makeatother

\tableofcontents



\section{Introduction}

 The original problem of Fourier restriction was first observed by E. Stein in the late~1960s,  as reported in  C. Fefferman's PhD thesis\ccite{Fefferman1}:
 the restriction of the Fourier transform of   functions $f\in L^p(\bR^n)$ to the unit sphere $\bS^{n-1}\subset\bR^n $ makes sense 
  when $p$ is close enough to~$1$ and $n>1$.
  We will recall the precise statement in  Theorem \ref{thm_1rest} and sketch its original proof in Section \ref{sec_sphere}.
  
  \smallskip
  
  Soon after this observation, E. Stein and P. Tomas in \cite{stein, Tomas} found the optimal range of indices $p$ for this restriction, and the sharpness of this result was proved with a counterexample due to A. Knapp, see \cite{strichartz}. In the meantime, A. Zygmund \cite{zygmund}  proved a discrete analogue restriction property in the context of the two-dimensional flat torus. 

 \smallskip

Over the subsequent decades, restriction problems have experienced an incredible effervescence and excitement. They remain a topical issue to this day. Generalised to other hyper-surfaces in Euclidean spaces and other settings, they are closely related to problems in Harmonic Analysis and Partial Differential Equations (PDEs), such as the Bochner-Riesz means, Strichartz estimates, and  the Kakeya conjecture   proved by R. Davies\ccite{Davies} in the two dimensional case in 1971, and   recently solved by  H. Wang and J. Zahl \cite{WZ} in the three-dimensional case. 
They are also related to many questions at the interface of Harmonic Analysis and  Number Theory; an instance of this is the Vinogradov conjecture (recalled in Section \ref{Vinogradov}), now a theorem  proved in 2015 independently by Bourgain-Demeter-Guth\ccite{BDG} and T. Wooley\ccite{Woo1, Woo}, using decoupling theory  for the former and for the latter the efficient congruencing. 
  For an in-depth presentation of these topics, we refer to the two milestone surveys on the subject by T. Tao \cite{Tao} and   L. Guth \cite{Guth},  as well as the texts of     C. Demeter\ccite{DemeterICM, DemeterBourgain} about decoupling methods and Bourgain's work in Fourier restriction.

 \smallskip

Restriction problems are also connected to many phenomena in Spectral Theory, and our aim is to explore geometric and spectral aspects of them. 
In Section \ref{sec_spectral}, we present the well-known interpretation of the original restriction problem on the sphere as a spectral property of the Laplace operator~$\Delta_{\bR^n}$ on~$\bR^n$. 
Replacing~$-\Delta_{\bR^n}$ with other positive operators allows us to formulate the problem of Fourier restriction  into a spectral problem that makes sense in other settings. 
In contrast with T. Tao \cite{Tao} and L. Guth \cite{Guth}'s surveys,
this text emphasises  this reformulation which  is especially relevant in  contexts where the Euclidean Fourier transform is unavailable. These settings 
include the  Riemannian manifolds with the Laplace-Beltrami operators,
but more generally any setting with a Laplace-type operator, such as 
graphs with the graph Laplacian or manifolds equipped with positive sub-elliptic operators. 
We will naturally recall the connection  
with cluster estimates  
understood as 
 $L^p-L^q$-bounds  for the spectral projectors of  a Laplace-type operator (or  its square root) in a window $[\lambda,\lambda+1]$ with  $\lambda\gg 1$ large  (see Section~\ref{subsec_cluster}).
The particular case of the Laplace-Beltrami operators on compact Riemannian manifolds has been  extensively  studied by C. Sogge\ccite{Sogge1988T, Sogge1988, Sogge2017}.

\smallskip  

A perhaps less understood question is to study cluster estimates or spectral restriction problems for sub-elliptic operators, for instance, sub-Laplacians, that is, the sum of squares of vector fields satisfying the H\"ormander condition\ccite{hormander4}. 
Exploring Stein's problem in sub-elliptic frameworks is  motivated by the wide range of ramifications of this field in
several parts of mathematics and  applied areas, such as    crystallography,  
  particle physics,    optimal
control, image processing, etc.\ccite{LeDonne}.
The prototype of sub-Laplacians is the canonical ones on  the Heisenberg groups; 
these groups  can be defined via the canonical commutation relations, known as CCR in quantum mechanics, and are at the confluence of many mathematical and physical fields.
Restriction properties for 
the  canonical sub-Laplacians on the Heisenberg groups were studied by D. M\"uller\ccite{Muller}. 
In this setting, there are no non-trivial solutions in the $L^p$-spaces for $p>1$, leading D. M\"uller   to reformulate the problem in anisotropic spaces, see Section\refer{sec_Liegroups}.

\medskip

This text is organised  as follows. 
In Section~\ref{sec_sphere}, we discuss the  first $L^2$-restriction theorem which was set on the unit sphere $\bS^{n-1}$ of $\bR^n$. 
We emphasise the crucial role of the Fourier transform of the canonical measure of $\bS^{n-1}$ and its decay at infinity.
In Section~\ref{sec_TS}, we present the extension of these questions to the setting of hyper-surfaces with non-vanishing Gaussian curvature in $\bR^n$
with optimal range of indices and exponents. 
This is known as the famous Tomas-Stein Theorem.
Again, the proofs rely fundamentally on the decay of 
 the Fourier transform of the induced measures on the hyper-surfaces with non-vanishing curvature. Indeed, this decay is the same as for the unit sphere $\bS^{n-1}$.
In Section~\ref{sec_spectral}, we reformulate the restriction estimates on the sphere in spectral terms. 
We show that they are equivalent to cluster estimates.
In Section~\ref{sec_Liegroups}, we discuss restriction and cluster estimates on Lie groups  and homogeneous domains, where few cases have been studied with  surprising results, such as M\"uller's results for the canonical sub-Laplacian on the Heisenberg group. 
Section~\ref{sec_open} is devoted to applications of restriction results and open problems. 
 The first part of this section will be mainly concerned with Strichartz estimates which have become a powerful tool in the study of nonlinear evolution equations involved in physics, quantum mechanics and general relativity. 
In the second part of Section~\ref{sec_open}, we briefly present restriction problems in  the  challenging  discrete framework, and its connections with  Number Theory.
Finally, in Appendix~\ref{sec_TT*}, we recall the $TT^*$ argument.

\medskip  We conclude this introduction with a comment on notation. In this paper, the letter $C$ will be used to denote  universal constants
which may vary from line to line. If we need the implied constant to depend on parameters, we shall indicate this by subscripts. We  also use the notation $A\lesssim B$   to
denote   bound of the form $A\leq C B$,   and $A \lesssim_\alpha B$ 
for $A\leq C_\alpha B$, where $C_\alpha$ depends only  on $\alpha$.

\bigbreak\noindent{\bf Acknowledgments.}
  The authors wish to thank    Julien Sabin   for    enlightening discussions  about  the refined restriction estimates. 
 
\section{The first $L^2$-restriction theorem: on the sphere}
\label{sec_sphere}

Historically, the first $L^2$-restriction problem was set on  the unit sphere
$$
\bS^{n-1} =\left \{x=(x_1,\ldots, x_n)\in \bR^n \colon |x|^2 = x_1^2+\ldots + x_n^2=1\right \},
$$ 
equipped with its canonical measure $\sigma_{\bS^{n-1}}$.

\subsection{Statement}
\label{subsec_statement+proof_sphere}

As the Fourier transform of a Schwartz function is Schwartz, its restriction to the unit sphere makes sense.
The first restriction theorem (Theorem \ref{thm_1rest} below) states that the Fourier transform  of a function in $L^p(\bR^n)$ restricts to an $L^2$-function on the unit sphere for~$p\in [1,4n /(3n+1))$.
To keep the notation consistent with the case of  more general settings, in what follows, we will distinguish $\bR^{n}$ and its dual $\wh \bR^{n}$, which is of course isomorphic to~$\R^{n}$ itself, 
and write~$\wh\bS^{n-1}$ for the unit sphere of $\wh \bR^{n}$.
By a classical density argument, it suffices to establish the 
\textit{a priori} estimate
\begin{equation}
\label{eq:estimesphere}
 \|\cF( f)|_{\wh\bS^{n-1}}\|_{L^2 (\sigma_{\wh\bS^{n-1}})}\leq C \|f\|_{L^p (\R^n)}\, ,
\end{equation}
 for all $f$ in  the Schwartz space $\cS(\R^n)$ and for a constant $C>0$ independent of $f$.
Above,~$\cF f$ denotes the Fourier transform of  $f$ on $\R^n$:
$$
\cF f (\xi ) = \widehat f(\xi)=\int_{\bR^n} e^{- 2\pi ix\cdot \xi} f(x) dx, \qquad \xi\in \wh\bR^n.
$$
Indeed, \eqref{eq:estimesphere} implies that the linear map $f\mapsto \cF f|_{\wh\bS^{n-1}}$ defined on $\cS(\bR^n)$ extends uniquely continuously $L^p(\bR^n)\to L^2(\wh\bS^{n-1})$.

 \begin{theorem}[First restriction theorem]
 \label{thm_1rest}
Let $n\geq 2$ and $p\in [1,4n /(3n+1))$.
The estimates in \eqref{eq:estimesphere} hold for any $f\in \cS(\bR^n)$. 
 \end{theorem}

Theorem \ref{thm_1rest} is far from being optimal.
For instance, a modification of the  argument below would allow us to  reach the end-point $p=4n / (3n+1)$ using Hardy-Littlewood-Sobolev inequalities.
However, the optimal end-point ($p_{TS}:=(2n+2)/(n+3)$) was proved later on with the Tomas-Stein Theorem\ccite{Tomas} (see Theorems \ref{thm_Tomas}
 and \ref{t:tsbase} below).

\subsection{Proof of Theorem \ref{thm_1rest}}
\label{subsec_pfthm_1rest}

We reproduce here the well-known argument due to E. Stein and C. Fefferman\ccite{stein, Fefferman1, Fefferman}.

\begin{proof}[Proof of Theorem \ref{thm_1rest}]
The proof relies on noticing the following equality valid for any~$f\in \cS(\bR^n)$
$$
\|\cF f\|_{L^2(\sigma_{\wh \bS^{n-1}})}^2 
= \int_{\bR^n}
f(x) \ \overline {f*\kappa (x)}\ dx,
\quad\mbox{where}\quad \kappa:= \cF^{-1}\sigma_{\wh \bS^{n-1}}.
$$
By Stein-Weiss\ccite{steinweiss} Section IV.3 (see also the monograph of Gel'fand-Shilov\ccite{GS}),  the convolution kernel $\kappa$ is explicitly 
given in terms of the Bessel function $J_\alpha$  of the first kind and of order $\alpha = (n-2)/2$ as
\begin{equation}
	\label{eq_kappa}
 \kappa(x) = \int_{\wh \bR^n} e^{2\pi ix\cdot  \xi} d\sigma_{\wh \bS^{n-1}}(\xi) = 
2\pi \frac{J_{(n-2)/2}(2\pi |x|)}{|x|^{(n-2)/2}}.
\end{equation}
It is well known that, for any $\alpha>-1$, 
the function  
$\frac{J_\alpha(z)}{(z/2)^{\alpha}}$ coincides on $\bR$ with a power series in $z^2$,  more precisely with $ {}_0F_1(\alpha+1,-z^2/4) / \Gamma(\alpha+1)$ using generalised hyper-geometric series and the Gamma function. Moreover, it satisfies the asymptotic bound 
$$
\exists C>0 \quad  \slash \quad \forall z\in \bR, \qquad 
|J_\alpha(z)|\leq C (1+|z|)^{-1/2}.
$$
Hence, $\kappa$ is a smooth function on $\bR^n$ satisfying 
\begin{equation}
\label{eq_estkappa}
	|\kappa(x)| \lesssim (1+|x|)^{-(n-1)/2},
\end{equation}
so that $\kappa\in L^q(\bR^n)$ if and only if $q>2n/(n-1)$.
By H\"older's inequality and then Young's convolution inequality, 
we have 
$$
\|\cF f\|_{L^2(\sigma_{\wh\bS^{n-1}})}^2 
\leq \|f\|_{L^p(\bR^n)} \|f*k\|_{L^{p'}(\bR^n)} 
\leq \|f\|_{L^p(\bR^n)}^2 \|k\|_{L^{q}(\bR^n)},
\quad \frac 1p =1-\frac 1{2q}.
$$
The conclusion follows. 
\end{proof}

\medbreak

The crucial step in the proof above is to show that the operator $\cP$ defined by
\begin{equation}
\label{eq_opP}
	\cP f = f*\kappa, \quad\mbox{where}\quad \kappa = \cF^{-1}\sigma_{\wh\bS^{n-1}},
\end{equation}
or equivalently via 
$$
\widehat {\cP f} = \cF f \ d\sigma_{\wh\bS^{n-1}}, 
$$
is bounded $L^p(\bR^n)\to L^{p'}(\bR^n)$ whenever $p\in [1,4n /(3n+1))$.
This is in fact an early instance of the $TT^*$ argument  (see Appendix~\ref{sec_TT*}) as 
we may write 
$$
\cP = TT^*, \quad\mbox{i.e.}\quad \forall f\in \cS(\bR),\qquad
\cP f = f * \kappa = TT^*f ,
$$
where $T^*$ is the actual restriction operator $f\mapsto \cF f|_{\wh \bS^{n-1}}$ whose boundedness $L^p\to L^2$ we wish to prove, and $T$ its formal dual given by $Tg = \cF^{-1} (g d\sigma_{\wh \bS^{n-1}})$. 
By the $TT^*$ argument, the boundedness of one of the following operators $$
\cP: L^p(\bR^n)\to L^{p'}(\bR^n), 
\qquad 
T\colon L^2(\wh \bS^{n-1}) \to L^{p'}(\bR^n), 
\qquad 
T^*\colon L^{p}(\bR^n) \to L^2(\wh \bS^{n-1})
$$
implies the boundedness of the others.
These boundedness properties may be expressed equivalently with the following {\it a priori} estimates valid for $f\in \cS(\bR^n)$ and $g\in C^\infty (\wh \bS^{n-1})$: 
  \begin{align*}
 \|\cP f\|_{L^{p'}(\bR^n)}&\leq C_p \|f\|_{L^{p}(\bR^n)},\\
 \|Tg\|_{L^{p'}(\bR^n)}= \|\cF^{-1} (g d\sigma_{\wh \bS^{n-1}})\|_{L^{p'}(\bR^n)}&\leq \sqrt C_p \|g\|_{L^2(\wh \bS^{n-1})},
\\
\|T^*f\|_{L^2(\wh \bS^{n-1})}= 
\|\cF f|_{\wh \bS^{n-1}}\|_{L^2(\wh \bS^{n-1})}
&\leq \sqrt C_p \|f\|_{L^{p}(\bR^n)}.
\end{align*}

\subsection{Tomas' improvement}
\label{subsec_Tomas}

In \cite{Tomas}, P. Tomas improved the range of $p$ for which the estimates in~\eqref{eq:estimesphere} hold:

\begin{theorem}[Tomas]
 \label{thm_Tomas}
Let $n\geq 2$ and $p\in [1,p_{TS})$ with 
$$
p_{TS}:=\frac{2n+2}{n+3}.
$$
The estimates in \eqref{eq:estimesphere} hold for any $f\in \cS(\bR^n)$. 
 \end{theorem}

\begin{proof}[Sketch of the proof of Theorem \ref{thm_Tomas}
 for $p\in [1,p_{TS})$]
The ideas in P. Tomas' proof \cite{Tomas}  rely on a dyadic decomposition and interpolating between the boundedness $L^1\to L^\infty$  and $L^2\to L^2$  for each dyadic piece. 
These two types of boundedness are proved using (respectively) the estimates in \eqref{eq_estkappa} for $\kappa = \cF^{-1}\sigma_{\wh \bS^{n-1}}$ and 
\begin{equation}
\label{eq_bSn-1dim}
\sigma_{\wh \bS^{n-1}}(\wh \bS^{n-1} \cap B(\xi_0, r)) = \cH^{n-1} (\wh \bS^{n-1} \cap B(\xi_0, r)) \sim r^{n-1};
\end{equation}
above, $r>0$ is small enough, $\xi_0\in \wh \bS^{n-1}$, $B(\xi_0,r)$ denotes the Euclidean ball about~$\xi_0$ with radius~$r>0$  and $\cH^{n-1}$ the Hausdorff measure of dimension $n-1$.
 For more details, see Section 7 in\ccite{Wolff}.
\end{proof}

\medbreak

In \cite{stein}, E. Stein shows that   the index $p_{TS}$ is achieved, using complex interpolation methods.

\smallskip

Let us also emphasise that a counterexample attributed to Knapp (see\ccite{strichartz},  Lemma~3)  shows that 
the estimates in \eqref{eq:estimesphere} do not hold for any $p>p_{TS}$.
It is given as follows: consider  $g_{\delta}$ ($\delta>0$ very small) the characteristic function of a spherical cap 
$$
\wh C_{\delta}:=\{x\in \wh\bS^{n-1}\colon |x\cdot e_{n}|<\delta \}\,.$$
One can prove that, as~$\delta \to 0$,  
$$
\|g_{\delta}\|_{L^{2}(\sigma_{\wh \bS^{n-1}})}\sim  \delta^{(n-1)/2},\qquad \| \cF^{-1}(g_{\delta}) \|_{L^{p'}( \R^{n})}\geq C \delta^{n-1}\delta^{-(n+1)/p'}\,,
$$ 
hence the estimate for some constant $C>0$
$$
\| \cF^{-1}(g_{\delta}) \|_{L^{p'}( \R^{n})}\leq C \|g_{\delta}\|_{L^{2}(\sigma_{\wh \bS^{n-1}})}\, ,
$$
can hold only if $p'\geq (2n+2)/(n-1)$, that is to say  if $p\leq p_{TS}$. 
By the $TT^*$ argument, the estimates in \eqref{eq:estimesphere} can hold only for $p\leq p_{TS}$.

\section{Tomas-Stein Restriction theorem on hyper-surfaces}
\label{sec_TS}
Given a hyper-surface~$\wh S \subset \wh\R^n$ endowed with a smooth measure $\sigma$, the restriction problem, that  has been  introduced by E.-M. Stein in the seventies, asks for which pairs~$(p,q)$ an inequality of the form
\begin{equation}
\label{eq:estimepq}
\exists C>0 \quad \slash \quad \forall f\in \cS(\bR^n),\qquad
 \|\cF( f)|_{\wh S}\|_{L^q (\wh S, \sigma)}\leq C \|f\|_{L^p (\R^n)}\, ,
\end{equation}
holds. The classical Tomas-Stein theorem focuses on the case $q=2$. 
The measure $\sigma$ is any measure having a smooth and compactly supported density with respect to the induced measure $\sigma_{\wh S}$ on $\wh S$. 
Before stating the famous result, 
let us recall the definition of $\sigma_{\wh S}$ and its fundamental properties used in the proof of restriction theorems. 

\subsection{The induced measure on a hyper-surface}
\label{subsec_inducedmeas}

\subsubsection{Definition}
\label{subsubsec_definducedmeas}
We consider a hyper-surface~$S$ of $\R^n$ with $n\geq 2$.
Equipping  $\bR^n$ with its Euclidean structure, 
$S$ inherits a Riemannian structure, the metric being  obtained by restricting the Euclidean metric to the tangent space of $S$. By definition, the induced measure $\sigma_{S}$ is the volume measure on $S$ associated with this Riemannian metric.
This generalises the notion of `surface area' measures for surfaces in $\bR^3$. 

\begin{example}
Naturally, in the case of the unit sphere $\bS^{n-1}$, the induced measure $\sigma_{\bS^{n-1}}$ coincides with its canonical measure.   
\end{example}

\begin{example}
\label{ex_Sgraph}
	If the hyper-surface is given by the graph of a smooth function $\phi:\bR^{n-1}\to \bR$
	$$
	S = \left \{(x',x_n)\in \bR^{n-1}\times \bR \sim \bR^n \colon x_n =\phi(x')\right \},
	$$
	then the induced measure is given by:
	$$
	\int_S f d\sigma_{ S} =\int_{\bR^{n-1}} f(x',\phi(x')) \sqrt{1+|\nabla_{x'} \phi (x')|^2 } dx',
	$$
	for any  function $f:\bR^n\to \bC$ in the space~$C_c(\bR^n)$ of continuous functions with compact support. 
\end{example}
As a hyper-surface may be described locally as in Example \ref{ex_Sgraph}, it follows that the estimates in \eqref{eq_bSn-1dim} set in the context of the unit sphere generalise into
\begin{equation}
\label{eq_Sn-1dim}
\sigma_S(S \cap B(x_0, r))  \sim r^{n-1},
\end{equation}
for any $x_0\in S$ and $r>0$ small enough. 

\smallskip

For the subsequent applications of the restriction problems to  PDE's (see Section \ref{sec_open}), the most relevant examples are the following particular cases of Example \ref{ex_Sgraph}:
\begin{example}
\label{ex_parcone}
For the parabola
$$
S_{par}:= \left \{(x',x_n)\in \bR^{n-1}\times \bR \sim \bR^n \colon x_n =|x'|^2\right \}, 
$$
the induced measure is given by 
		$$
		\forall f\in C_c(\bR^n),\qquad 
	\int_{S_{par}} f d\sigma_{ S_{par}} =
	\int_{\bR^{n-1}} f(x',|x'|^2) \sqrt{1+|x'|^2 } dx'.
	$$
Although the cone, 
$$
S_{cone}:= \left \{(x',x_n)\in \bR^{n-1}\times \bR \sim \bR^n \colon x_n = |x'|\right \} 
$$
has a singularity at $0$, it is possible to define its  induced measure as being given by 
	$$
		\forall f\in C_c(\bR^n),\qquad 
	\int_{S_{cone}} f d\sigma_{ S_{cone}} =\sqrt{2 }
	\int_{\bR^{n-1}} f(x',|x'|)  dx'.
	$$
\end{example}

\subsubsection{Estimates for the Fourier transform of the induced measure}
A fundamental result in the analysis of restriction problems  is the following:
\begin{theorem}[\cite{stein}] \label{thm_estwhsigma} 
 Let $S$ be a smooth  hyper-surface in $\R^{n}$ with $n\geq 2$.
 We denote by~$\sigma_{S}$  the induced measure on $S$. 
 We fix  $\psi\in C_c^\infty (S)$ and consider the measure given by 
 $$
 d\sigma(x):= \psi(x) \, d\sigma_{S}(x).
 $$
 The Fourier transform of $\sigma$ is a smooth function $\wh \sigma\in C^\infty(\bR^n)$. Moreover, 
 if the Gaussian curvature of $S$ does not vanish at every point, then it  satisfies the following estimate:
 $$
 \exists C>0\quad \slash \quad  \forall \xi\in \bR^n, \qquad
 |\wh \sigma (\xi) |\leq C (1+|\xi|)^{-\frac{n-1}2}.
 $$
\end{theorem}

In the case of $S=\bS^{n-1}$ and $\psi=1$, we recognise the estimate in \eqref{eq_estkappa} used in the proof of the first $L^2$-restriction theorem (Theorem \ref{thm_1rest}).

\begin{proof}[Sketch of the proof of Theorem \ref{thm_estwhsigma}]
As $\sigma$ is a measure with compact support, its Fourier transform $\wh \sigma$ is a smooth function on $\bR^n$. 
The estimate is obtained by applications of stationary phase methods that we now sketch. 

We may assume that the support of $\psi$ is included in a chart for $S$. Moreover, after an orthogonal change of coordinates in $\bR^n$, we may assume that on the support of $\psi$,  $S$ is described by the graph of a function as in Example \ref{ex_Sgraph}. 
We then obtain
$$
	\widehat \sigma(\xi) =\int_{\bR^{n-1}} e^{-i \phi_\xi (x')} \psi_1(x') dx',
	\quad\mbox{where}\ 
\psi_1(x'):= \psi(x',\phi(x') ) \sqrt{1+|\nabla\phi|^2}(x'), 
$$
with phase  given by 
$$
\phi_\xi  (x') = (x',\phi(x'))\cdot  \xi = x_1\xi_1+\ldots +x_{n-1}\xi_{n-1} + \phi(x')\xi_n.
$$
We write $\xi = t \tilde \xi$ with $t\geq 0$ and $\tilde \xi\in \bS^{n-1}$, and consider $t$ large. 
If the direction $\tilde\xi $ is 
away from the South and North poles $(0,\ldots, 0, \pm 1)$, then 
the non-stationary phase yields readily
$|\widehat \sigma(\xi)|\lesssim_N t^{-N}=|\xi|^{-N}$,  for any $N\in \bN.$
If $\tilde \xi$ is near the South and North poles, then, since the Gaussian curvature of $S$ does not vanish, we may apply the  stationary phase to obtain $|\widehat \sigma(\xi)|\lesssim t^{-(n-1)/2} =  |\xi|^{-(n-1)/2}$.
\end{proof}

\medbreak

\subsubsection{Measure theoretic description}
\label{subsubsec_measT}
The definition in Section \ref{subsubsec_definducedmeas} from differential geometry may also be described in terms of objects in measure theory,  see the monographs of H. Federer\ccite{Federer} or P.  Mattila\ccite{Mattila}.
Indeed, 
the induced measure $\sigma_{S}$ coincides with the restriction to $S$ of the Hausdorff measure $\cH^{n-1}$ of dimension $n-1$:
$$
\sigma_{S} (A)= \cH^{n-1}(A),
$$
for any Borelian subset $A$ of $S$.
From elementary tools in measure theory and differential geometry, in particular Example \ref{ex_Sgraph}, it follows that, when $S$ is bounded, 
the restriction of~$\cH^{n-1}$ to $S$ may also be described as the weak limit of the measure 
$$
\cH^{n-1}(B\cap S)= \lim_{\epsilon\to 0} \frac{\vol\, B\cap S_\epsilon}{\vol\, S_\epsilon}, 
$$
for any Borelian subset~$B\subset \bR^n$; above,~$\vol$ denotes the Lebesgue measure of~$\bR^n$, and~$S_\eps$ is the~$\epsilon$-tubular neighbourhood of~$S$, that is, the set of points in~$\bR^n$ at distance less than~$\epsilon$ to~$S$.
Equivalently, 
$$
\sigma_S(A)=
\cH^{n-1}(A)= \lim_{\epsilon\to 0} \frac{\vol\, A_1\cap S_\epsilon}{\vol\, S_\epsilon}, 
$$
for any  Borelian subset~$A\subset S$, having denoted by~$A_1$  its $\epsilon'$-tubular neighbourhood with~$\epsilon'=1$. 
In the particular case where $S$ is compact with~$\cH^{n-1}(S)=1$, the above formula simplifies as:
\begin{equation}
	\label{eq_inducedmeascompactS}
	\sigma_S(A)=
\cH^{n-1}(A)= \lim_{\epsilon\to 0} \frac{\vol\, A_1\cap S_\epsilon}{2\epsilon}.
\end{equation}
We will see in Section \ref{subsubsec_cP} that  the Fourier restriction problem on the sphere may be reformulated into a spectral problem, with an operator defined as the spectral analogue to~\eqref{eq_inducedmeascompactS}. 

\subsection{Statement of the Tomas-Stein Restriction Theorem}

Generalising the results obtained in the case of the unit sphere (see Theorem \ref{thm_1rest}) and with optimal range including the Tomas-Stein index (see Section \ref{subsec_Tomas}), Stein\ccite{stein} proved the following result:

\begin{theorem}[\cite{Tomas, stein}] 
\label{t:tsbase} 
 Let $\wh S$ be a smooth  hyper-surface in $\wh \R^{n}$ with $n\geq  2$.
 We denote by $\sigma_{\wh S}$  the induced measure on $\wh S$. 
 We fix  $\psi\in C_c^\infty (\wh S)$ and consider the measure given by 
 $$
 d\sigma(\xi):= \psi(\xi) \, d\sigma_{\wh S}(\xi).
 $$
 If the Gaussian curvature of $\wh S$ does not vanish,
 then the estimates in \eqref{eq:estimepq} hold 
for $q=2$ and every $1 \leq p\leq p_{TS}=(2n+2)/(n+3)$.
Consequently, the  linear map $f\mapsto \cF f|_{\wh S}$ defined on $\cS(\bR^n)$ extends uniquely continuously $L^p(\bR^n)\to L^2(\wh    S)$.
\end{theorem}

The methods of proofs in the case of the sphere (see Section \ref{subsec_Tomas}) generalise to hyper-surfaces with non-vanishing Gaussian curvature. 
In particular, they rely on the estimates for the Fourier transform inverse of $\sigma$ in Theorem \ref{thm_estwhsigma}.

 \smallskip  
 Arguing along the same lines as in Section \ref{sec_sphere}, we deduce that the following estimates hold true, for any hyper-surface  $\wh S$ satisfying the hypothesis of Theorem \ref{t:tsbase}:
 \begin{align}
\label{dual3} \|\cP f\|_{L^{p'}(\bR^n)}&\leq C_p \|f\|_{L^{p}(\bR^n)}\\
\label{dual} \|Tg\|_{L^{p'}(\bR^n)}= \|\cF^{-1} (g d\sigma_{\wh S})\|_{L^{p'}(\bR^n)}&\leq \sqrt C_p \|g\|_{L^2(\wh S)}
\\
\label{dual2}
\|T^*f\|_{L^2(\wh S)}= 
\|\cF f|_{\wh S}\|_{L^2(\wh S)}
&\leq \sqrt C_p \|f\|_{L^{p}(\bR^n)}\, .
\end{align} 

 Naturally, the canonical example of a surface satisfying the hypotheses of Theorem~\ref{t:tsbase}   is the unit sphere $\wh \S^{n-1}$ of $\wh \bR^n$, see Section~\ref{sec_sphere}.  
Other examples of hyper-surfaces  such as the paraboloid whose Gaussian curvature does not vanish  (see Example~\ref{ex_parcone}) have been  the subject of a number of works. 
\smallskip The hypothesis
of curvature  in Theorem \ref{t:tsbase} can be  weakened:    similar results are  possible for hyper-surfaces that are not flat but with vanishing Gaussian curvature, such as  the cone  (see Example~\ref{ex_parcone}) which differs from the sphere and the paraboloid, since it has a vanishing principal curvature. In this case, the range of $p$ is smaller depending on the order of tangency of the surface to its tangent space. However, the hypothesis of non-flatness of the hyper-surface is mandatory. Indeed, consider the function  $f:\R^{n}\to\R$ defined by
\beq
\label{cexR}\ds  f(x)= \frac{e^{- |x'|^2}} {1+|x_1|}\, \virgp\qquad x=(x_1,x')\in \R^{n}.
\eeq
Then, one can  check that $f$ belongs to $L^p(\R^n)$, for all $p>1$, but its Fourier transform does not admit a restriction on the hyper-plane $\wh S$ of $\wh \R^n$ defined by $ \wh S=\{\xi \in \wh  \R^n \, /\, \xi_1=0 \}$.

\section{Spectral viewpoint}
\label{sec_spectral}

In this section, we reformulate the Fourier restriction problem on the unit sphere in terms of the spectral decomposition of the Laplace operator. This allows   us to avoid relying on  the Fourier transform and to make the link with the cluster estimates.

\subsection{Spectral description of the operator $\cP$}

In this section, we go back to the first~$L^2$-restriction theorem on the unit sphere (Section \ref{sec_sphere}) in order to start reformulating the problem spectrally. 

We recall that the canonical measure $\sigma_{\bS^{n-1}} $ on $\bS^{n-1}$ coincides with the measure used in  the polar change of coordinates:
\begin{equation}
	\label{eq_polar}
\forall f\in \cS(\bR^n),\qquad 
\int_{\bR^n}f(x)dx = \int_{r=0}^{+\infty} \int_{\bS^{n-1}} f(rx) d\sigma_{\bS^{n-1}} r^{n-1} dr.
\end{equation}

\subsubsection{The operator $\cP$}
\label{subsubsec_cP}
The operator $\cP$ defined in \eqref{eq_opP} may be understood 
in terms of the spectral decomposition of the (self-adjoint extension of the) rescaled  Laplace operator
\begin{equation}
\label{eq_spectral}
-(2\pi)^{-2} \Delta = \int_0^\infty \lambda dE_\lambda
\end{equation}
of~$\bR^n$; for spectral decomposiiton see for instance \cite{RudinFA} or  the book of  M. Lewin\ccite{Lewin}  and the references therein.
The factor $2\pi$ is irrelevant in our analysis. It appears so that  the spectral decomposition $E_\lambda$ interacts well with our (choice of normalisation for the) Fourier transform on $\bR^n$.
The spectral projectors $E_\lambda$ are explicitly given  via the Fourier transform as
\begin{equation}
	\label{eq_spectralELaplaceRn}
	E[a,b] f = \cF^{-1} \left ( 
\bone_{[a,b]}
( |\cdot|^2)\cF(f)\right), 
\qquad a<b, \quad f\in L^2(\bR^n).
\end{equation}
In this particular case, spectral multipliers in $-(2\pi)^{-2}\Delta$, that is,  
$$
 m(-(2\pi)^{-2}\Delta) = \int_0^\infty m(\lambda) dE_\lambda, \quad m\in L^\infty (\bR),
$$
coincide with radial Fourier multipliers:
\begin{equation}
\label{eq_mult_Fvsrad}
	  m(-(2\pi)^{-2}\Delta) f= \cF^{-1} \left( m(|\cdot|^2)\cF(f)\right).
\end{equation}
  
We observe:
\begin{lemma}
\label{lem_cP}
	 The operator $\cP$ defined via \eqref{eq_opP} coincides with 
	 $2\lim_{\varepsilon\to 0}\frac 1\varepsilon E[1,1+\varepsilon]$
in the sense that 
$$
\forall f,g\in \cS(\bR^n),\qquad 
(\cP f,g)_{L^2(\bR^n)} = 2\lim_{\varepsilon\to 0}
\frac 1\varepsilon (E_{[1,1+\varepsilon]} f,g)_{L^2(\bR^n)}.
$$
\end{lemma}

We may write
$$
\cP = 2\lim_{\varepsilon\to 0}\frac 1\varepsilon E[1,1+\varepsilon]
 = 2\partial_{\lambda=1} E[0,\lambda],
$$
since $E[0,\lambda_2]-E[0,\lambda_1] = E[\lambda_1,\lambda_2]$, for any $0\leq \lambda_1\leq \lambda_2$.

\begin{proof}[Proof of Lemma \ref{lem_cP}]
Using the description of the spectral projector $E[1,1+\varepsilon]$ in terms of the Fourier transform and a polar change of coordinates (see \eqref{eq_spectralELaplaceRn}
 and \eqref{eq_polar} respectively), we obtain:
\begin{align*}
	\frac 1\varepsilon(E[1,1+\varepsilon] f,g)_{L^2(\bR^n)}
&=\frac 1\varepsilon ( 
\bone_{[1,  1+\varepsilon]}(|\cdot|^2) \cF f ,\cF g )_{L^2(\bR^n)}\\
&=\frac 1\varepsilon\int_{\{\xi\colon 1 \leq |\xi|^2\leq 1+\varepsilon\} }
\cF f(\xi) \ \overline{\cF g}(\xi) \ d\xi
\\&=\frac 1\varepsilon\int_1^{ \sqrt{1+\varepsilon}}
\biggl(\int_{\wh \bS^{n-1}} 
\cF f(r\omega) \ \overline{\cF g}(r\omega) 
d\sigma_{ \wh \bS^{n-1}} (\omega)\biggr)\, r^{n-1}dr
\end{align*}
whose limit as $\varepsilon\to 0$ is 
$$
\frac 12 \int_{\wh \bS^{n-1}}
\cF f(\omega) \ \overline{\cF g}(\omega) 
d\sigma_{ \wh \bS^{n-1}}(\omega) 
=\frac 12 (f * \cF^{-1} \sigma_{\wh \bS^{n-1}} , g)_{L^2(\bR^n)}
=
\frac 12 (\cP f, g)_{L^2(\bR^n)}.
$$
\end{proof}

We may equivalently define $\cP$ as 
$$
\cP =- 2\lim_{\varepsilon\to 0}\frac 1\varepsilon E_{[1-\varepsilon,1]},
\quad\mbox{or as}\quad
 \cP = \lim_{\varepsilon\to 0}\frac 1{\varepsilon} E[1-\varepsilon,1+\varepsilon].
$$
With the latter expression, we see that  the  construction for  $\cP$ is the spectral analogue of the geometric construction of the surface measure, see 
Section \ref{subsec_inducedmeas}, especially
\eqref{eq_inducedmeascompactS}.

\medskip

We may similarly define
$$
\cP_{\lambda_0}:= 2\partial_{\lambda=\lambda_0} E[0,\lambda] =2\lim_{\varepsilon\to 0}\frac 1\varepsilon E[\lambda_0,\lambda_0+\varepsilon]= \lim_{\varepsilon\to 0}\frac 1\varepsilon E[\lambda_0-\varepsilon,\lambda_0+\varepsilon],
$$
for any $\lambda_0\in \bR$, 
with $\cP_1=\cP$ for the case of $\lambda_0=1$. 

\subsubsection{Dilations properties}
The properties of $\cP_{\lambda_0}$ may be deduced from 
the properties of~$\cP$ by exploiting  the 2-homogeneity of 
the differential operator $\Delta$,
$$
\mbox{i.e.}\qquad \forall f\in C_c^\infty(\bR^n),\qquad \forall r>0,\qquad 
\Delta (f\circ \delta_r) = r^2 (\Delta f)\circ \delta_r,
$$
for the 
dilations $\delta_r\colon x\mapsto rx$, $r>0$ on $\bR^n$.
Indeed, the 2-homogeneity of $\Delta$ implies
for  any~$m\in L^\infty(\bR)$
\begin{equation}
	\label{eq_hom}
 \forall r>0, \qquad \forall f\in L^2(\bR^n),\qquad 
m(-(2\pi)^{-2}\Delta) (f\circ \delta_r ) = (m(-r^2(2\pi)^{-2}\Delta) f)\circ \delta_r, 
\end{equation}
and in particular for $m=\bone_{[a,b]}$, 
$$
E[a,b]=\bone_{[a,b]}(-(2\pi)^{-2}\Delta)
\quad\mbox{and}\quad  E[r^{-2}a,r^{-2}b]=\bone_{[a,b]}(-r^2(2\pi)^{-2}\Delta).
$$
Hence, we obtain for any $\lambda_0>0$:
$$
\cP_{\lambda_0} f = \frac 1{\lambda_0} (\cP (f\circ \delta_{\sqrt{\lambda_0}}^{-1})  )\circ \delta_{\sqrt{\lambda_0}}.
$$
Indeed in view of\refeq{eq_hom} and in the sense of Lemma\refer{lem_cP}, we have
\begin{align*}
	\cP (f\circ \delta_{\sqrt{\lambda_0}}^{-1})
	&=2\lim_{\varepsilon\to 0}\frac 1\varepsilon (E[\lambda_0,\lambda_0+\varepsilon \lambda_0] f)\circ \delta_{\sqrt{\lambda_0}}^{-1}\\
	&= 2\lim_{\varepsilon'\to 0}\frac {\lambda_0}{\varepsilon'} (E[\lambda_0,\lambda_0+\varepsilon '] f)\circ \delta_{\sqrt{\lambda_0}}^{-1}\\
	&= \lambda_0 \ (\cP_{\lambda_0}f)\circ \delta_{\sqrt{\lambda_0}}^{-1}.
\end{align*}

Using the properties of the dilations on $L^p$-norms, that is,
$$
\forall f\in \cS(\bR^n), \qquad \forall r>0, \qquad \forall p\in [1,\infty],\qquad 
\| f\circ\delta_r \|_{L^p(\bR^n)} = r^{-\frac n p} \|f\|_{L^p(\bR^n)}, 
$$
we obtain for any $p,q\in [1,\infty]$,
\begin{equation}
\label{eq_cPlambda0cP}
	\|\cP_{\lambda_0} \|_{\sL(L^p(\bR^n),L^q(\bR^n))}
=\lambda_0^{-1+\frac n2 (\frac 1p -\frac 1q)} \|\cP \|_{\sL(L^p(\bR^n),L^q(\bR^n))}.
\end{equation}
Recalling that with $\lambda=\eps^{-1}$
$$ E[\lambda,\lambda+1] f = (E[1,1+\eps] (f\circ \delta_{\sqrt{\lambda}}^{-1})  )\circ \delta_{\sqrt{\lambda}}\, ,$$ we also obtain  (for $\lambda=\eps^{-1}$)
\begin{equation}
\label{eq_EepsElambda}
\|E[1,1+\eps] \|_{\sL(L^p(\bR^n),L^q(\bR^n))}
=\lambda^{-\frac n2 (\frac 1p -\frac 1q)} \|E[\lambda,\lambda+1]\|_{\sL(L^p(\bR^n),L^q(\bR^n))}.
\end{equation}

\subsection{Spectral reformulation of the restriction property}

The properties above imply  readily that the $L^p-L^q$-boundedness of $\cP$
is in fact equivalent to certain spectral estimates:
\begin{lemma}
\label{lem_equiv_spectralest}
	Let $p,q\in [1,\infty]$. 
	The operator $\cP$ is bounded $L^p(\bR^n)\to L^q(\bR^n)$
	if and only if 
	\begin{equation}
		\label{eq_cluster_est}
			\exists C>0 \quad \slash \qquad \forall \lambda\geq 1 ,\qquad
	\|E[\lambda,\lambda+1]\|_{\sL(L^p(\bR^n),L^q(\bR^n))} 
	\leq C \lambda^{-1 +\frac n2 (\frac 1p -\frac 1q)}.
	\end{equation}
\end{lemma}
\begin{proof}
 Since 
	$\cP_{\lambda_0}= 2\partial_{\lambda=\lambda_0} E[0,\lambda]$,
	we have, at least formally,
	$$
	E[\lambda,\lambda+1] = \frac 12\int_\lambda^{\lambda+1} \cP_{\lambda_0} d\lambda_0.
	$$ 
	Therefore, 	if $\cP$ is bounded $L^p(\bR^n)\to L^q(\bR^n)$, then we have,   for all $\lambda \geq 1$,
\begin{align*}
&\|E[\lambda,\lambda+1]\|_{\sL(L^p(\bR^n),L^q(\bR^n))}
	\leq
	\frac 12 \int_\lambda^{\lambda+1} \| \cP_{\lambda_0}\|_{\sL(L^p(\bR^n),L^q(\bR^n))}
 d\lambda_0 	\\
 &\qquad \qquad \leq \frac 12 \| \cP\|_{\sL(L^p(\bR^n),L^q(\bR^n))}
 \int_\lambda^{\lambda+1}  \lambda_0^{-1+\frac n2 (\frac 1p -\frac 1q)}  d\lambda_0 	\asymp \lambda^{-1+\frac n2 (\frac 1p -\frac 1q)},
\end{align*}
having used  \eqref{eq_cPlambda0cP}.
Conversely, combining Lemma \ref{lem_cP} with  \eqref{eq_EepsElambda},
\begin{align*}
	\|\cP \|_{\sL(L^p(\bR^n),L^q(\bR^n))}
	&= 
	\lim_{\eps\to 0} \frac 2 \eps  \|E[1,1+\eps ]\|_{\sL(L^p(\bR^n),L^q(\bR^n))}
	\\&= 2\lim_{\lambda\to +\infty} \lambda^{1-\frac n2 (\frac 1p -\frac 1q)}  \|E[\lambda,\lambda+1 ]\|_{\sL(L^p(\bR^n),L^q(\bR^n))} 
\end{align*} 
is bounded when \eqref{eq_cluster_est} holds.
\end{proof}

\subsubsection{Equivalences to the restriction property}
Let us summarise the equivalences  to the restriction property for the sphere that we have obtained in the following statement: 
\begin{corollary}
\label{corlem_equiv_spectralest}
Let $p\in [1,+\infty)$.
The following are equivalent:
\begin{enumerate}
	\item The restriction operator  $f\mapsto \cF f|_{\wh \bS^{n-1}}$ is bounded  $L^p(\bR^n) \to L^2(\wh \bS^{n-1})$.
	\item  The operator $g\mapsto  \cF^{-1} (g d\sigma_{\wh \bS^{n-1}})$ is bounded $L^2(\wh \bS^{n-1}) \to L^{ p'}(\bR^n)$.	
	\item The operator $\cP$ is bounded $L^p(\bR^n)\to L^{p'}(\bR^n)$.
	\item There exists $C>0$ such that for any $\lambda\geq 1$, we have:
$$	
	\|E[\lambda,\lambda+1]\|_{\sL(L^p(\bR^n),L^{p'}(\bR^n))} 
	\leq C \lambda^{-1 +n (\frac 1p-\frac 12)},
$$
 or equivalently
$$
\|E[\lambda,\lambda+1]\|_{\sL(L^p(\bR^n),L^{2}(\bR^n))} 
	\leq \sqrt C \lambda
	^{-\frac 12 +\frac n2 (\frac 1p-\frac 12)},
$$
or equivalently
$$
\|E[\lambda,\lambda+1]\|_{\sL(L^2(\bR^n),L^{p'}(\bR^n))} 
	\leq \sqrt C \lambda^{-\frac 12 +\frac n2 (\frac 1p-\frac 12)}.
$$
\end{enumerate}	
Moreover, by the Tomas-Stein theorem (Theorem \ref{t:tsbase}), the sharp range for Part (1) above is 
$1\leq p \leq p_{TS}=(2n+2)/(n+3)$. 
\end{corollary}
\begin{proof}[Proof of Corollary \ref{corlem_equiv_spectralest}]
The equivalence follows readily from 
	Lemma \ref{lem_equiv_spectralest}
 together with the comments on the  proof of Theorem \ref{thm_1rest} 
at the end of Section \ref{subsec_pfthm_1rest}
 and the application of the $TT^*$ argument to $E[\lambda,\lambda+1]=E[\lambda,\lambda+1]^2$.
\end{proof}

 Instead of considering $-\Delta$, 
we can obtain similar estimates  as in Part (4), but for~$\sqrt{-\Delta}$ instead of $-\Delta$ by  modifying the arguments above  (a more direct reasoning using the Fourier transform is explained at the beginning of Chapter 5 in \cite{Sogge2017}):
\begin{corollary}
\label{cor_clustersqrtDelta}
For any $1\leq p \leq p_{TS}$, there exists $C>0$ such that  for any $\lambda\geq 1$, we have
$$
\|\bone_{[\lambda,\lambda+1]}(\sqrt{-\Delta})\|_{\sL ( L^p(\bR^n),L^2(\bR^n))}
\leq C  \lambda ^{\delta(p)}, \qquad 
\delta(p) := n  \left(\frac 1p -\frac 12\right)-\frac 12.
$$
We have equivalently, 
$$
\|\bone_{[\lambda,\lambda+1]}(\sqrt{-\Delta})\|_{\sL ( L^2(\bR^n),L^{p'}(\bR^n))}
\leq C  \lambda ^{\delta(p)}, 
$$
or 
$$
\|\bone_{[\lambda,\lambda+1]}(\sqrt{-\Delta})\|_{\sL ( L^p(\bR^n),L^{p'}(\bR^n))}
\leq C^2  \lambda ^{2\delta(p)}.
$$
\end{corollary}
 
\begin{proof}[Proof of Corollary \ref{cor_clustersqrtDelta}]
The equivalence comes from  the $TT^*$ argument. Hence, it suffices to prove the $L^p-L^{p'}$-inequality. 
We write
\begin{align*}
\bone_{[\lambda,\lambda+1]}(\sqrt{-\Delta})
 	&= \bone_{[\lambda^2,(\lambda+1)^2]}(-\Delta)
 	= E[\lambda^2,(\lambda+1)^2] 
 	\\&= 
 	 \int_{\lambda^2}^{(\lambda+1)^2}
 	 \partial_{\lambda_0} E[0,\lambda_0]
 	  d\lambda_0
 	  =\frac 12 \int_{\lambda^2}^{(\lambda+1)^2}
 	  \cP_{\lambda_0} d\lambda_0,	
\end{align*}
so using \eqref {eq_cPlambda0cP}
 	\begin{align*}
 &\|\bone_{[\lambda,\lambda+1]}(\sqrt{-\Delta})\|_{\sL ( L^p(\bR^n),L^{p'}(\bR^n))}
 	\\&\quad \leq\frac 12 \int_{\lambda^2}^{(\lambda+1)^2}
 	 \| \cP_{\lambda_0} \|_{\sL ( L^p(\bR^n),L^{p'}(\bR^n))}
 d\lambda_0\\
 &\qquad
 =\frac 12 \|\cP \|_{\sL(L^p(\bR^n),L^2(\bR^{p'}))}
 \int_{\lambda^2}^{(\lambda+1)^2}
\lambda_0^{-1+\frac n2 (\frac 1p -\frac 1{p'})} 
 d\lambda_0		\\
 &\qquad   \asymp \|\cP \|_{\sL(L^p(\bR^n),L^{p'}(\bR^2))} \lambda^{n(\frac 1p -\frac 1{p'}) -1}.
 	\end{align*}
 	We recognise $n(\frac 1p -\frac 1{p'}) -1 = 2\delta(p)$. 
 	We conclude with  $\cP$ being $L^p\to L^{p'}$-bounded for~$1\leq p\leq p_{TS}$ (see Corollary \ref{corlem_equiv_spectralest}). 
\end{proof}
Proceeding as for the proof of Corollary \ref{corlem_equiv_spectralest}, we can show that  any of the equivalent estimates in  Corollary \ref{cor_clustersqrtDelta} is equivalent to 
the Tomas-Stein restriction theorem in Theorem~\ref{t:tsbase}.
Moreover, the range in $p$ is sharp. 
 
\subsection{Links with the cluster estimates} 
\label{subsec_cluster}

We can consider similar estimates as in Part~(4) of Corollary~\ref{corlem_equiv_spectralest}   or in Corollary \ref{cor_clustersqrtDelta}
 in other contexts, 
that is, the following estimates 
\begin{equation}
	\label{eq_cluster0}
		\exists C>0 \quad \slash \qquad \forall \lambda\geq 1, \qquad
	\|\bone_{[\lambda,\lambda+1]}(\cL)\|_{\sL(L^p(M),L^{q}(M))} 
	\leq C \lambda^{\gamma_{p,q}},
\end{equation}
for some indices $p,q\in [1,\infty]$ and exponents $\gamma_{p,q}\in \bR$; 
expressing these estimates  necessitates 
a non-negative  self-adjoint operator  $\cL$ on $L^2(M)$ 
and
a measurable space $M$ equipped with a sigma-finite positive measure.
This may be  studied in settings different from $\bR^n$, and 
no  meaning needs to be attached to the operator $\cP$ or to the restriction operator. In particular,  the existence of a Fourier transform  is not required. 
Furthermore, 
the space $M$ need not be even a manifold but e.g. a tree or a graph with $\cL$ being the graph Laplacian.

Roughly speaking, the estimates in \eqref{eq_cluster0} measure the density or the distribution of $\cL$-eigenvalues in a window of length 1 in the $\sL(L^p,L^q)$-norm.
When the spectrum is discrete, taking $q=2$, they yield $L^p$-estimates of $\cL$-eigenfunctions, 
and furthermore, estimates for the $L^p-L^2$-boundedness of the corresponding spectral projectors, 
see \eqref{eq_estPimu} below in the case of a compact Riemannian manifold. 
Variations of the estimates in \eqref{eq_cluster0} may be considered, for instance for other spaces than~$L^r(M)$, $r=p,q$, 
or different windows from $[\lambda,\lambda+1]$,  see e.g. the review of  P. Germain\ccite{Germain}.
These types of estimates are often called \textit{cluster estimates}.

\smallskip The most notable example of cluster estimates are set on a compact Riemannian manifold $M$ equipped with the volume element.
In this case, 
the Laplace-Beltrami operator~$\cL$ on a compact Riemannian manifold has compact resolvent, so its spectrum is discrete and each eigenvalue has finite multiplicity. 
The spectral decomposition of~$\sqrt\cL$ may be described very concretely in terms of a chosen orthonormal basis of $\cL$-eigenfunctions.
The corresponding cluster estimates have been proved by C. Sogge 
 on the sphere and then on any compact Riemannian manifold\ccite{Sogge2017}:

\begin{theorem}
\label{thm_cluster}
	Let $(M,g)$ be a compact Riemannian manifold of dimension $n\geq 2$.
	We keep the same notation for the Laplace-Beltrami operator $\cL$ and its self-adjoint extension on $L^2(M)$.
	The operator $\sqrt{\cL}$ satisfies the  cluster estimates 
	as in \eqref{eq_cluster0} with $q=2$, $p\in [1,2]$ and
	$$ 
	\gamma_{p,2}:= \left\{ \begin{array}{ll}
 	 n \left(\frac 1p -\frac 12\right)-\frac 12 & \mbox{if}\ 1\leq p \leq p_{TS}=\frac{2n+2}{n+3},\\
 	\frac{n-1}2 \left(\frac 1p-\frac 12 \right)& \mbox{if}\ p_{TS} \leq p \leq 2.
 \end{array}
\right.
	$$
	Furthermore, these estimates are sharp  in the sense that, for each $p\in [1,2]$,  there exist  compact Riemannian manifolds for which  the exponent $\gamma_{p,2}$  is sharp. 
\end{theorem}

Theorem \ref{thm_cluster} means that for any $p\leq 2$, there exists a constant~$C>0$ such that 
\begin{equation}
	\label{eq_thm_cluster}
	\forall \lambda\geq 1, \quad  \forall f\in L^p(M), \qquad
	\|\bone_{[\lambda,\lambda+1]}(\sqrt{ \cL}) f\|_{L^2(M)} 
	\leq C \lambda^{\gamma_{p,2}} \|f\|_{L^p(M)},
\end{equation}	
with exponents $\gamma_{p,2}$ given in the statement.
In particular, 
the critical index  is again~$p_{TS} = \frac {2n+2}{n+3}$.
The regimes~$p\in [1,p_{TS}]$ and~$p\in [p_{TS},2]$ are sometimes called {\it point-focusing} and {\it geodesic-focusing} respectively.
The exponents in the point-focusing regime, that is, 
$$
\gamma_{p,2}=n  \left(\frac 1p -\frac 12\right)-\frac 12 =\delta(p), 
\qquad 1\leq p \leq p_{TS},
$$
 are the same as for the cluster estimates on $\bR^n$ for $\sqrt{-\Delta}$, see Corollary \ref{cor_clustersqrtDelta}.
 
 \smallskip Before sketching elements of the proof, 
 we observe that, by the $TT^*$ argument (see Appendix~\ref{sec_TT*}), the cluster estimates in \eqref{eq_thm_cluster} are equivalent to 
 \begin{align}
 \label{eq_thm_cluster_dual2p'}
 	 \forall g\in L^2(M), \qquad
	\|\bone_{[\lambda,\lambda+1]}(\sqrt{ \cL}) g\|_{L^{p'}(M)} 
	&\leq C \lambda^{\gamma_{p,2}} \|g\|_{L^2(M)}\\
	\label{eq_thm_cluster_dualpp'}
	\forall f\in L^p(M), \qquad
	\|\bone_{[\lambda,\lambda+1]}(\sqrt{ \cL}) f\|_{L^{p'}(M)} 
	&\leq C^2 \lambda^{2\gamma_{p,2}} \|f\|_{L^p(M)}.
 \end{align}
 
 \begin{proof}[Elements of the proof of Theorem \ref{thm_cluster}] The first idea of the proof consists in establishing\refeq{eq_thm_cluster} for  $p=1,2$, using  functional analysis. Indeed, we readily have that 
\begin{equation}
	\label{eq_thmL2}
\|\bone_{[\lambda,\lambda+1]}(\sqrt{ \cL})\|_{\sL(L^2(M))}\leq 1\, ,
\end{equation}
and combining   the spectral theory for unbounded self-adjoint operators on compact Riemannian manifold  and  the point-wise sharp Weyl law, it follows that \begin{equation}
	\label{eq_thmL1}
\|\bone_{[\lambda,\lambda+1]}(\sqrt{ \cL})\|_{\sL(L^1(M),L^2 (M))}\lesssim \lam^ { \frac  {n-1}2}\, .
\end{equation}
Let us give a few more details for the proof of \eqref{eq_thmL1}: for any function $\psi:\bR\to \bC$, 
 we have
 $$
\|\psi(\sqrt{ \cL})\|_{\sL(L^1(M),L^2 (M))}
\leq \sup_{x\in M} |K_{\psi}(x,x)|\, ,
$$
using functional properties of the integral kernel $K_{\psi}$  of the operator $\psi(\sqrt{ \cL})$ and the spectral decomposition of $\sqrt{ \cL}$.
The estimate for the kernel for  $\psi = 1_{[\lambda,\lambda+1]} = 1_{[0,\lambda+1]}-1_{[0,\lambda[}$ may be found in \cite[Lemma 4.2.4]{Sogge2017} or can be deduced from  the point-wise sharp Weyl law in compact manifolds  \cite{SZ}. This shows \eqref{eq_thmL1}.

\smallskip 
The second idea which will complete  the proof of the theorem, 
is to interpolate  the two estimates \eqref{eq_thmL2}-\eqref{eq_thmL1} corresponding respectively to $p=2$ and  $p=1$ with the estimate for the case $p=p_{TS}$.  However in this setting, the study of the case $p=p_{TS}$ requires to first replace 
$\bone_{[\lambda,\lambda+1]}(\sqrt \cL)=\bone_{[0,1]}( \sqrt \cL -\lambda)$  with a smooth approximant~$\chi(\sqrt \cL - \lambda)$ for some  $\chi\in \cS(\bR)$, 
and then to perform a deep study of certain oscillatory integral operators. See \cite{Sogge2017} for the complete proof.  
  \end{proof}
 
 A consequence of the cluster estimates in the form \eqref{eq_thm_cluster_dual2p'} is that if $g$ is an eigenfunction for $\cL$ with eigenvalue $\mu$, then 
$g$ is also an eigenfunction for $\sqrt \cL$ but with eigenvalue $ \lambda=\sqrt \mu$, and we have
$$
\|\bone_{[\lambda,\lambda+1]}(\sqrt{ \cL}) g\|_{L^{p'}(M)} 
 =
\|g\|_{L^{p'}(M)}\leq C \lambda^{\gamma_{p,2}} \|g\|_{L^2(M)}.
$$
Consequently, 
the orthogonal projector 
 $\Pi_\mu :=\bone_{\{\mu\}} ( \cL)=\bone_{\{\sqrt \mu\}} (\sqrt \cL)$  onto the $\mu$-eigenspace for $\cL$ is bounded $L^2\to L^{p'}$ with operator norm
 $$
  \|\Pi_\mu\|_{\sL(L^2(M), L^{p'}(M))}\leq  C \mu^{\gamma_{p,2}/2}.	
$$
 By duality, we also have:
  \begin{equation}
 \label{eq_estPimu}
  \|\Pi_\mu\|_{\sL(L^{p}(M),L^2(M))}\leq  C \mu^{\gamma_{p,2}/2}.\end{equation}

\section{Restriction theorems on  Lie groups and homogeneous domains}
\label{sec_Liegroups}

Here, we discuss the restriction and cluster estimates 
 in the setting of Lie groups and homogeneous domains. 
 In the compact case, the case of a Laplace operator invariant under the action of the group is covered by C. Sogge's results on cluster estimates on compact manifold (see Theorem \ref{thm_cluster}). However, more can be said in the case of the torus (see Section \ref{des-case}) and 
a different question may be asked for instance for nilpotent Lie groups equipped with sub-Laplacians.

\subsection{A few words on the analysis on Lie groups and homogeneous domains}

Here, we consider (real, finite dimensional) Lie groups and their homogeneous domains. Recall that  by definition,  homogeneous domains are quotients of  Lie groups by  closed subgroups.

\subsubsection{First examples of Lie groups and homogeneous domains}
\label{subsubsec_exGhomdom}
The Euclidean space $\bR^n$ or the torus $\bT^n$ are naturally commutative Lie groups. 
The torus may also be viewed as a homogeneous domain since it  is the quotient $\bT^n = \bR^n / \bZ^n$.
Intuitively, 
homogeneous domains are manifolds that are highly symmetrical. Mathematically, this means that a large group acts on them. 
Examples include the unit sphere
$$
\bS^{n-1} =\left \{x=(x_1,\ldots, x_n)\in \bR^n \colon |x|^2 = x_1^2+\ldots + x_{n-1}^2+ x_n^2=1\right \}, \ n\geq 2,
$$ 
and  
the hyperbola
$$
S_{hyperbol} =\left \{x=(x_1,\ldots, x_n)\in \bR^n \colon   x_1^2+\ldots + x_{n-1}^2 - x_n^2=1\right \},\ n\geq 2.
$$ 
 The natural groups acting on them are, respectively, the groups $O(n)$ and $O(n-1,1)$ of orthogonal and Lorentz transformations.
Recall that orthogonal and Lorentz transformations are, by definition, the linear transformations preserving the quadratic forms~$|x|^2 = x_1^2+\ldots + x_{n-1}^2+ x_n^2$ and $x_1^2+\ldots + x_{n-1}^2 - x_n^2$ (respectively). 

\subsubsection{Definition and further examples}

By definition, a Lie group $G$ is  a set equipped with compatible structures as a smooth manifold and as a group. In other words, the smooth manifold $G$ is also equipped with the group operations of multiplication $(x,y)\mapsto xy$ and inverses $x\mapsto x^{-1}$ 
which are assumed to be smooth ($G\times G\to G$ and $G\to G$ respectively).

\begin{example}
The  general linear group ${\rm GL}(n,\bR)$  of  $n \times n$ invertible matrices over the reals is a Lie group. For $n>1$, it is non-commutative. 
\end{example}

In fact, any Lie group is locally isomorphic to a matrix group, that is,  a closed subgroup of ${\rm GL}(\bR,n)$ for $n$ large enough. 

\subsubsection{Types of Lie groups by examples}
Matrix groups such as the orthogonal group~${\rm O}(n)$, the unitary group, or the symplectic group together with their non-compact counterparts (e.g. $O(n-1,1)$) belong to the realm of {\it semisimple} or more generally {\it reductive} Lie groups\ccite{Hall,Helgason,varad}. They are crucial in describing  natural  homogeneous domains such as the sphere (viewed as the quotient~$O(n) / O(n-1)$ or equivalently   $SO(n) / SO(n-1)$) or the hyperbola $S_{hyperbol}$  (viewed as the quotient $O(n-1,1) / O(n-1)$). Moreover, their Lie algebras are fundamental objects in mathematical physics, especially for particle physics.

\smallskip Another important class of Lie groups consists of  the {\it solvable} Lie groups.
\begin{example}
\label{ex_ax+b}
The group of affine transformations on $\bR$ given by $x\mapsto ax+b$ with $a\in \bR\setminus\{0\}$ and $b\in \bR$ is a solvable Lie group that is not nilpotent.  
It is often called  the $AX+B$ group, and may be also realised  as the matrix group:
$$\left\{\left(
\begin{matrix}
a & b\\
0& 1	
\end{matrix}\right) \colon a\in \bR\setminus\{0\}, b\in \bR \right\}.
$$	
\end{example}

A significant subclass among solvable Lie groups consists of the {\it nilpotent} Lie groups. 
The latter are often assumed connected and simply connected without a real loss of generality \cite{CorwinGreenleaf}. In this case, they can  be described as the closed subgroups of the matrix group  
$$
\left(\begin{matrix}
1 & *&*\\
0& \ddots &*\\
0& 0&1	
\end{matrix}\right).
$$
Equivalently, they may be viewed as the group $\bR^n$ equipped with a polynomial law\ccite{Puk}.

\begin{example}
\label{ex_Hn}
	Perhaps the most famous examples of nilpotent Lie groups are the Heisenberg groups $\bH_d$, with  underlying manifold $\bR^{2d+1}$.
The three dimensional Heisenberg group may be realised as the matrix group
\begin{equation}
\label{eq_Heis_matrix}
\bH_1\sim 
\left\{
\left(\begin{matrix}
	1&a&c\\
	0&1&b\\
	0&0&1
\end{matrix}\right)
\colon a,b,c\in \bR\right\}.	
\end{equation}
Because of its relation to complex analysis, the Heisenberg group $\bH_d$ is often realised as~$\bR^{2d+1}$ equipped with the following group law:
\begin{equation}
	\label{eq_Hnlaw}
(Y,s)(Y',s') = (Y+Y', s+s'+ 2 (\langle y, \eta'\rangle -\langle y', \eta\rangle)),
\end{equation}
where $(Y,s)=(y,\eta,s)$ and $(Y',s')=(y',\eta',s')$ are in $\bR^d\times \bR^d\times \bR\sim \bR^{2d+1}$.
\end{example}

\subsubsection{Technical definitions}
\label{subsubsec_technicalDef}
Let us briefly present the technical definitions of the notions sketched above:
\begin{definition}
	A (connected) Lie group $G$ is {\it reductive, semisimple, solvable or nilpotent} when its Lie algebra $\fg$ is so. 
	
\begin{enumerate}
	\item A Lie algebra is {\it nilpotent} if any finite number $s$ of iterated Lie brackets vanishes.  In fact, there exists a smallest number $s$ valid for any  nested Lie brackets; this is called the step of nilpotency of $\fg$. 
	
	Equivalently, a Lie algebra $\fg$ is nilpotent when its lower central series eventually reaches zero. The lower central series is defined with  
	$\fg^1=\fg$ and recursively $\fg^{k+1}=[\fg,\fg^k]$ being the linear span of elements  $[V,W]$,  $V \in \fg$ and $W \in \fg^k$.
	\item  A Lie algebra is {\it solvable} when its derived series eventually reaches zero. The derived series is defined with 
	 $\fg^{(0)}=\fg$ and recursively $\fg^{(k+1)}=[\fg^{(k)},\fg^{(k)}]$ being the linear span of elements $[V,W]$, $V,W\in \fg^{(k)}$.
	 \item A Lie algebra is {\it semi-simple} when it has no-non zero abelian ideals.  
	 \item A Lie algebra is {\it reductive} when it can be written as the direct sum of its centre with a semi-simple Lie algebra. 
\end{enumerate}	 
\end{definition}

The classifications of semi-simple Lie algebras and of the homogeneous domains of the semisimple Lie groups (symmetric domains)
are well known \cite{Helgason}.
This is a beautiful and elegant theory connected to the Killing form, Dynkin diagrams, root space decompositions, the Serre presentation, the theory of highest weight representations,
the Weyl character formula for finite-dimensional representations, etc. 
By contrast, if it is possible to classify the nilpotent Lie algebra of dimension $\leq 6$, in dimension 7, it is possible to identify families of non-isomorphic Lie algebras parametrised by real numbers
\cite{Magnin}. 

Certain types of nilpotent Lie groups are relevant to control theory and sub-Riemannian geometry:
\begin{definition}
	A Lie algebra is {\it stratified} when it can be decomposed as 
	$$
	\fg = \fg_1 \oplus \ldots \oplus \fg_s 
	\quad\mbox{with}\quad [\fg_i ,\fg_j] =\fg_{i+j},
	$$
	with the convention that $\fg_k=0$ for $s>k$. 
\end{definition} 
In this case, the Lie algebra is nilpotent, and the corresponding connected simply connected Lie group is said to be stratified.
For instance, the Heisenberg group is naturally stratified.
After a choice of an inner product or a basis on  the first stratum $\fg_1$, a stratified  group is said to be {\it Carnot}.  

\subsubsection{Solvmanifolds}
The homogeneous domains for solvable and nilpotent Lie groups are called 
solvmanifolds and 
nilmanifolds respectively.

\begin{example}
The M\"obius strip is a solvmanifold as it is the quotient of the group $AX+B$ (see Example \ref{ex_ax+b}) by the subgroup formed by the transformations corresponding to $a=\pm 1$ and $b=0$.
\end{example}

\begin{example}
	The canonical three dimensional Heisenberg nilmanifold is   the matrix group in \eqref{eq_Heis_matrix}  quotiented by the matrix subgroup corresponding to integer coefficients. 
	This may be described as an $S^1$-bundle over $\bT^2$.
\end{example}

Not every nilpotent Lie group may be quotiented into a compact nilmanifold (see
\cite[Chapitre 5]{CorwinGreenleaf}
and references therein). However, it is possible on the Heisenberg group and more generally nilpotent Lie groups with two-step of nilpotency, such as in\ccite{fermfischer3}.

\subsubsection{Analysis on compact Lie groups}
Global and Harmonic analysis on 
compact (hence reductive) Lie groups  and their homogeneous domains is now well-understood \ccite{Helgason}. 
Moreover, 
many results initially obtained in these settings have proven to be generalisable to compact manifolds without an underlying group structure. 
 A notable example of this phenomenon is for instance the cluster estimates (see Theorem \ref{thm_cluster}): although it was first proved by C. Sogge for the Laplace-Beltrami operator on the sphere using techniques of harmonic analysis\ccite{Sogge1988}, Sogge has generalised the result to elliptic pseudo-differential operators\ccite{Sogge2017}.
Naturally, many sophisticated questions may still use the underlying structure of Lie groups or homogeneous domains, for instance in the case of the torus
(see Section \ref{des-case}) or  general compact Lie groups, see e.g.\ccite{fischerJFA15} or\ccite{shaoC}.
 
 \smallskip  
The global analysis of non-compact Lie groups is  more intricate for reasons explained in the next paragraph.

\subsubsection{Convolution operators}
A Lie group $ G$ is equipped with positive measures that are invariant under translation, meaning multiplication on the left or on the right by any group element. These measures are known as the (left or right) Haar measures.  
Up to a constant, a left (resp. right) Haar measure is unique. 
This leads to the definition of~$L^p$-spaces on~$G$, as well as  convolution functions and operators.

\smallskip On $\bR^n$,  the analysis of 
convolution operators
is well-understood through the theory of singular integral operators 
developed by   A.-P. Calder\'on and A. Zygmund\ccite{CZ}  in the decades around the 1950s, and also the early works of Elias Stein\ccite{steinSIO}.
This theory extends readily to operators with integral kernels (not necessarily of convolution) and can be adapted to 
other settings\ccite{coifmanweiss}, including  Lie groups and their homogeneous domains. 
However, this adaptation depends on the behaviour of the Haar measure. The growth of the Haar measure of a (connected) Lie group is either polynomial or exponential - see\ccite{guivarch}  for a precise  statement  of this property.
Examples of Lie groups with polynomial growth of the volume 
include compact and nilpotent Lie groups, while non-compact reductive Lie groups and certain solvable Lie groups such as the $AX+B$ group have exponential growth of the volume. 
 The adaptation of the Calder\'on-Zygmund theory to Lie groups has become well understood in the  polynomial growth case\ccite{alexo,DtER}, 
 while the case of exponential growth is technically challenging \cite{HebischSteger}.

\subsubsection{Invariant Laplace operators}
Other operators that are naturally considered on Lie groups are differential operators invariant under left or right translations.
To fix the ideas, let us consider invariance under left translations. 
The Lie algebra $\mathfrak g$ of the Lie group $G$ may be viewed as the Lie algebra of 
left-invariant   vector fields.

\begin{example}
\label {ex_horizvect}
The following vector fields on $\bR^{2d+1}$
$$
 \cX_j:=\partial_{y_j} +2\eta_j\partial_s\andf \Xi_j:= \partial_{\eta_j} -2y_j\partial_s, \quad j=1,\ldots, d, 
$$
are invariant under the left translations of the Heisenberg group $\bH_d$.
Together with $\partial_s$, they form the canonical basis of the  Lie algebra of  $\bH_d$.
They satisfy the following Canonical Commutation Relations mentioned in the introduction:
$$
[\cX_j, \Xi_j] = -4\partial_s, 
\qquad 
[\cX_j, \Xi_k] = 0 \ \mbox{for} \ j\neq k.
$$
(The non-zero factor $-4$ is irrelevant.)
\end{example}
 
 The underlying vector space of the Lie algebra $\mathfrak g$ of $G$ 
   is naturally isomorphic to the tangent space of $G$ at its neutral element.
Any left-invariant differential operator may be written as a non-commutative polynomial in a basis of $\mathfrak g$, in a unique way once an ordering of the basis is fixed\ccite{Helgason}.

In the case of a compact Lie group $G$, 
it is natural to consider the left-invariant Laplace operator  defined as follows:
 we  equip the Lie algebra $\mathfrak g$ with a $G$-invariant inner product, yielding a left-invariant Riemannian metric on $G$ as well as a left-invariant Laplace-Beltrami operator given by 
$-(X_1^2+\ldots +X_n^2)$ where $X_1,\ldots, X_n$ are left-invariant vector fields that form an orthonormal basis of $\mathfrak g$.
A similar construction is often considered on symmetric spaces (i.e. the suitable compact and non-compact quotients of semisimple  Lie groups, see\ccite{Helgason})  such as the  sphere $\bS^{n-1}$ and the hyperbola $S_{hyperbol}$. 

\subsubsection{Sub-Laplacians}
An important class of differential operators related to the analysis on Lie groups are sub-Laplacians. By  sub-Laplacians, here, we mean   the sum of squares~$X_1^2+\ldots +X_{n'}^2$ of  vector fields $X_1, \ldots, X_{n'}$ of a manifold $M$ that generate the entire tangent space at every point of $M$  by linear combinations of iterated nested commutators. This hypothesis  on the family of vector fields is often referred to as  the H\"ormander condition, 
as L. H\"ormander\ccite{hormander4} showed in 1967 that 
it implies the hypoellipticity of sub-Laplacians,
see also\ccite{street}.
This  was the motivation behind the research programme led by Folland and Stein and their collaborators
around the 1980s, see e.g.\ccite{Folland, RLS},
on subelliptic operators modelled via nilpotent Lie
groups, in particular on Carnot groups (see Section \ref{subsubsec_technicalDef}) which are naturally equipped with a canonical sub-Laplacian.
This is related to non-holonomic geometries, 
often called  {\it sub-Riemannian} since the 1990's, with applications  in optimal
control, image processing and biology (eg human vision), see for instance the book by 
A. Agrachev, D. Barilari, and U. Boscain\ccite{ABB19}.  
It has rich motivations and ramifications in
several parts of mathematics: the analysis of hypoelliptic PDEs and  stochastic (Malliavin) calculus but also geometric measure theory, metric group theory, etc., see e.g.\ccite{LeDonne}.

\subsubsection{$L^2$-theory}
\label{subsubsec_L2T}
The starting point of many questions in analysis is an $L^2$-decomposition, be it as a hypothesis in  singular integral theory or from Fourier analysis or as a consequence of the self-adjointness of an  operator.

\smallskip  
On $\bR^n$ and $\bT^n$, $L^2$-decompositions are obtained readily according to  the Plancherel formula   via the Euclidean Fourier transform and the Fourier series. The latter are strongly related to their group structure and can be generalised on many topological groups. 
The case of compact groups was proved by F. Peter and H. Weyl\ccite{PeterWeyl} in 1927 (see also\ccite{steinTopics}), 
and provides a decomposition of the $L^2$-space analogous to the one provided by Fourier series: the sum is over all the irreducible representation modulo equivalence.
In the 1960's, 
this was further generalised by J. Dixmier\ccite{dixmier} to a very large class of topological groups (more technically: locally compact, unimodular, type I)  and provides a Fourier transform in terms of the (unitary irreducible) representations of the group. It provides not only a decomposition of the $L^2$-space but also an understanding of Fourier multipliers, or equivalently, the operators that are invariant under (e.g. left) translations and bounded on $L^2$.
Although J. Dixmier's results are very general and abstract, 
they can be made very concrete on groups whose  representation theory is very explicit, 
such as  semisimple Lie groups via weight theory, 
and  nilpotent Lie groups with the 
orbit method\ccite{kirillov}.

\smallskip Considering operators that are self-adjoint will also provides 
an $L^2$-decomposition via their spectral decomposition $E(\lambda)$, see e.g.\ccite{RudinFA}.
This leads to  the definition of spectral multipliers, and  is of particular interest for sub-Laplacians (see e.g.\ccite{alexo}).

\subsection{Restriction theorem on the Heisenberg group}\label{Restriction-Heisenberg}
\subsubsection{The setting}
 Realising the Heisenberg group $\bH_d$ as in Example~\ref{ex_Hn} with group law given in \eqref{eq_Hnlaw}, we define the following differential operator:
\begin{equation}
	\label{eq_sublapHd}
\D_{\H_d}u:=\sum_{j=1} ^d(\cX_j^2u+\Xi_j^2u) \, ,
\end{equation}
 where the  vector fields~$\cX_j$ and~$\Xi_j$ 
 were defined in Example \ref {ex_horizvect}.
 Note that 
 this choice of vector fields on the first stratum of the Heisenberg Lie algebra equipp $\bH_d$ with its natural structure of Carnot group (see Section \ref{subsubsec_technicalDef}) for which $\D_{\H_d}$ is its canonical sub-Laplacian.
 The operator $\D_{\H_d}$ is homogeneous of order two for the anisotropic dilations \begin{equation}
\label {dil}\delta_r (Y,s) = (rY, r^2 s)\, .\end{equation}
These dilations are automorphisms of the group $\bH_d$ - unlike the isotropic ones.

\subsubsection{M\"uller's result}
\label{Muller}
The  operator $-\D_{\H_d}$ is non-negative and  essentially self-adjoint  on~$C_c^\infty(\bH_d) \subseteq L^2(\bH_d)$, 
and it is also invariant under left-translation. 
From the $L^2$-theory viewpoint (see Section\refer{subsubsec_L2T}),
this implies that its spectral decomposition $E(\lambda)$ may be expressed in terms of the Fourier decomposition of the Heisenberg group, see\ccite{fischerRuzhansky} and\ccite{bcdh,bgx}.
Consequently,~$E[\alpha,\beta]$ may be described with special functions connected to the representation theory of~$\bH_d$. 
D. M\"uller has given explicit formulae for these   in  \cite{Muller}.
In the same paper, he studies 
the restriction theorem on the sphere of the Heisenberg group: as emphasised in Section \ref{subsubsec_cP}, 
this amounts  in studying the analogue of the operator $2\cP$  
formally defined as 
$$
2\cP =  \partial_{\lambda=1} E[0,\lambda].
$$
In particular, he shows that this operator is a  convolution operator whose convolution kernel is a tempered distribution  formally given by 
\beq
\label {defG} G(Y,s)=   \frac {2^d} {\pi^{d+1}}  \sum_{m\in\N^d} \frac1{(2|m|+d)^{d+1}} \cos\Big(\frac{s}{2|m|+d}\Big) \cW\Big(m,m,1,\frac{Y}{\sqrt{2|m|+d}}\Big)\, , \eeq
with $Y=(y, \eta)\in \R^d\times \R^d$ and
where $\cW$ denotes the Wigner transform of the (renormalized) Hermite functions  
 \beq
\label {defW}
\cW(m, m, \lam,Y)
:=\int_{\R^d} e^{2i\lam\langle \eta,z\rangle} H_{m,\lam}(y+z) H_{m,\lam} (-y+z)\,dz\,.
 \eeq
Here~$H_{m,\lam}$ stands for the renormalized Hermite function on~$\R^d$, namely (for further details, see \cite{GS, MOS1966}) $$H_{m,\lam} (x):= |\lam|^{\frac d 4} H_m(|\lam|^{\frac 12} x)\,,$$ with~$ \suite H m {\N^d}$ the Hermite orthonormal basis of~$L^2(\R^d)$ given by the eigenfunctions of the harmonic oscillator:
$$
-(\D -|x|^2) H_m= (2|m|+d) H_m\, ,
$$
specifically \beq
\label {hermite}
H_m:=  \Bigl(\frac 1 {2^{|m|} m!}\Bigr) ^{\frac 12} \prod_{j=1}^d  \big(-\partial_j H_0+ x_jH_0\big)^{m_j}  \, ,
\eeq
with $H_0(x):= \pi^{-\frac d 4} e^{-\frac {|x|^2} 2}$,  $m!:= m_1!\dotsm m_d!\,$ and $\,|m|:= m_1+\cdots+m_d.$

\medskip The formula in \eqref{defG} should be compared with its analogue $\kappa$ on $\bR^n$ given in \eqref{eq_kappa}.
From~\eqref{defG}, we observe that the behaviour in $s$ is almost transport-like. In fact, D. M\"uller in  the proof of \cite[Proposition 3.1]{Muller} constructs a Schwartz function $f\in \cS(\bH_d)$ such that 
$$
2\cP f (Y,s)= e^{-\frac{|Y|^2}d} \cos \frac s d.
$$
This is an obstruction to any $L^p\to L^q$-boundedness for $\cP$. 
It  is also related to the Bahouri-G\'erard-Xu counter-example to the Schr\"odinger propagation for $\D_{\bH_d}$, see  page 96  in \cite{bgx}.
However, D. M\"uller proves also the boundedness of $\cP$ 
for the anisotropic Lebesgue  spaces
$$
L^p_YL^q_s=L^p(\bR^{2d}\colon L^q(\bR)),
$$
namely the boundedness $L^p_Y L^1_s \to L^{p'}_Y L^\infty_s$:
$$
\|\cP f\|_{L^{p'}_Y L^\infty_s} \leq C \|f\|_{L^{p}_Y L^1_s} , 
\qquad 1\leq p <2.
$$
This has been re-interpreted in terms of a Fourier restriction theorem for the group Fourier transform of the Heisenberg group\ccite{BBG}, then adapted to suitable hyper-surfaces in the setting of $\R \times\H_d$, see further discussion in Section\refer{Strichartz}.

\smallskip  Note  that  restriction issues has been also investigated  in a few other   sub-elliptic  frameworks, see for instance V. Casarino-P. Ciatti\ccite{CC},        H.~Liu-M.~Song\ccite{LS11, LS16} and  D.~Barilari-S.  Flynn\ccite{BF}.


%
%
%
%
%
%
%
%

\section{Related issues and open questions}
\label{sec_open}

Restriction problems are at the intersection of  many areas of mathematics such as  Harmonic Analysis, Spectral and Geometry Theories,   with a broad range of applications covering  PDEs, Number Theory, Probability Theory, etc.
This subject is strongly linked with  many questions that are still largely  opened, hence giving a complete survey  to all these connexions  is beyond the scope of the present text. However, we will  present in the first part of this section some  applications of Tomas-Stein  inequalities  (in  well-known frameworks) in the fields of PDEs and   Spectral Theory.  Then, we  will close the text  by highlighting the  interplay between restriction problems and Number Theory.

\subsection{Related questions in harmnoic analysis}\label{related} 
 Fourier restriction problems are deeply connected to  two other conjectures central to Euclidean harmonic analysis,  and  known as  Kakeya and Bochner-Riesz.
 The  Kakeya conjecture states that {\it each Kakeya set in $\R^n$, that is to say a set containing a unit line in every direction, has Minkowski and  Hausdorff dimension equal to $n$}.
 This was proved by R. Davies\ccite{Davies} in the two dimensional case in 1971, and very recently   by  H. Wang and J. Zahl\ccite{WZ} in the three dimensional case;
 for insights about the Kakeya conjecture, we refer the reader to\ccite{DemeterICM, Tao, WZ}. 
 
 \smallskip  
 The Bochner-Riesz conjecture (recalled below) is a significant problem in harmonic analysis, focusing on the convergence and boundedness in $L^p$ of the Bochner-Riesz means.
 It has been proved by L. Carleson and P. Sj\"olin \cite{{CalesonSjolin}} in  the two dimensional case, 
but remains open   in higher dimensions.
  As regards to its link with  restriction problems, we refer the interested reader to the survey of  T. Tao\ccite{Tao} and the references therein.
 
 \smallskip  
 The Bochner-Riesz conjecture arose in the study of the ball multipliers at frequency $R$
  $$ 
  \bone_{B(0, R)}(D)f:=\cF^{-1} (\bone_{B(0,R)} \widehat f), \quad f\in \cS(\bR^n).
  $$
  In other words, 
$$
  \bone_{B(0, R)} (D) f (x)= \int_{B(0, R)} e^{2\pi ix\cdot \xi} \widehat f(\xi) d\xi, 
  \qquad   x\in \bR^n.   
  $$
According to Plancherel's formula, the family $\bone_{B(0, R)} (D)$ converges to $f$ in $L^2 (\R^n)$ as $R$ tends to infinity.
By contrast, the famous result of C. Fefferman\ccite{Fefferman-B} implies its  unboundedness in $L^p (\R^n)$
as soon as $n \geq 2$ and $p \neq 2$. The   failure of $L^p$ convergence of the ball multipliers has raised the same question with  the ball multipliers being replaced by the following family of  smoother Fourier multiplier operators, known as the
Bochner-Riesz multipliers:  
$$ 
S^\delta_Rf := \cF^{-1} (m_{\delta,R}(|\cdot|^2) \widehat f), \quad f\in \cS(\bR^n), 
$$
where 
$$
m_{\delta,R} (\lambda):= \left (1-\frac{\lambda}{R^2}\right )_+^\delta, 
$$
with $x_+:= \max (x, 0) $ denoting the positive part of $x$. 
In other words, 
$$
S^\delta_R f (x) =\int_{\bR^n} \left(1-\frac{|\xi|^2}{R^2}\right)_+^\delta \widehat f(\xi) e^{2\pi i x\cdot \xi} d\xi , \qquad  x\in \bR^n.
$$ 
 The  Bochner-Riesz conjecture asks {\it if    
 $\delta>0$ and   $n |\frac 1 2- \frac 1 p| - \frac 1 2 < \delta$, 
 then $S^\delta_R f$ converges to~$f$ in $L^p (\R^n)$, as $R$ tends to infinity? } Note that such condition on $p$  is known to be necessary, according to a counter example where the symbol of $S^\delta_R$ is divided into a large number of Knapp's examples, see E.-M. Stein\ccite{steinknapp}.

\smallskip  
We  observe that the Riesz means $S^\delta_R$ are radial Fourier multipliers, and can therefore be 
 equivalently described as a spectral multiplier in $-(2\pi)^{-2}\Delta$,
see \eqref{eq_mult_Fvsrad}:
$$
S^\delta_R f =m_{\delta,R}(-(2\pi)^{-2}\Delta).
$$
Hence, a similar question may be asked for any positive self-adjoint operator in other contexts.
The most natural is certainly the case of the Laplace operator on the torus, yielding to 
the Bochner-Riesz means for Fourier series.
The case of the canonical sub-Laplacians on the Heisenberg groups has also been considered \cite{mauceri,MullerBR}.

\subsection{Applications to evolution equations}\label{App}
In the present paragraph, we aim at briefly providing the different strategies  used to establish Strichartz estimates, which are another face of the Tomas-Stein  theorem.

Strichartz estimates  date back to the 70s through the founding paper of R. Strichartz \cite{strichartz}. They have become   an efficient tool for analysing many phenomena in physics, biology, fluid mechanics, general relativity, etc. 

\subsubsection{Origin of the Strichartz estimates}\label{Strichartz}
Here, we consider the Schr\"odinger equation, that have been introduced  in the context of quantum mechanics by  E.   Schr\"odinger   in 1925,   
\begin{equation}
	\label{eq_Schro}
	\left\{
\begin{array}{rcl}
\ds i\partial_t u -  \D u & =  & 0\\
{ u}{}_{|t=0} &= & u_0 \in L^2(\R^n)\, .
\end{array}
\right.
\end{equation}
Based on  standard arguments, one can easily check that the solution of the above Cauchy problem can be written as follows  
\begin{align}
\label{eq:str0}
u(t,x)&=\int_{\wh \bR^n}e^{i(x\cdot \xi + t|\xi|^{2})}\widehat{u}_{0}(\xi)d\xi \\ 
\label{eq:str0bis}&= \frac {{\rm
e}^{i \frac {|\cdot|^2} {4t} }} {(4 \pi i t)^\frac {n} 2} \star u_0\, .	
\end{align}
In \cite{strichartz}, R. Strichartz pointed out that  Formula \eqref{eq:str0} can be interpreted as the restriction of the Fourier transform on the paraboloid~$\wh S$  in the space of frequencies $\wh \R^{n+1}=\wh \R \times \wh \R^{n}$,  defined as
$$\wh S:=\big\{(\alpha,\xi)\in \wh \R\times \wh\R^{n}\mid \alpha=|\xi|^{2} \big\}\, . $$ 
Combining the $L^2\to L^{p'}$-bound formulation of the restriction theorem (see \eqref{dual}) with  scaling arguments, he deduced   the following bounds for space-time norms of  the solution~$u$ to the Schr\"odinger equation  by means of the $L^2$-norm of the initial data $u_0$: \begin{equation}\label{eq:strlambda00}
\|u\|_{L^{\frac{2n+4}{n}}(\R,
 L^{\frac{2n+4}{n}} (\R^{n}))}\leq C \|u_{0}\|_{L^{2}(\R^{n})}\, .
\end{equation}
Since then, these type   of estimates, which have  appropriate analogue results   for the wave equation and several systems involved in  fluid mechanics, were  christened with the name of Strichartz estimates.

 \subsubsection{Dispersion and Strichartz estimates}
 \label{subsubsec_dispersion_strichartz}
In mathematics, a dispersive evolution corresponds to the phenomena     that waves with different frequencies move
 at different velocities in time. For the Schr\"odinger equation in~\eqref{eq_Schro}, 
 this can be seen in 
 the Fourier representation\refeq{eq:str0} of its solution.
 Applying Young's convolution inequality to \eqref{eq:str0bis} implies that this solution satisfies the so-called {\it dispersive} estimate:
 \beq
\label {dispSR0}
\|u(t,\cdot)\|_{L^\infty(\R^{n})} \leq \frac 1 {(4\pi|t|)^{\frac n 2}} \|u_0\|_{L^1(\R^{n})}\,, \, \, \forall t \neq 0.
\eeq
  Commonly, a dispersive estimate corresponds to a pointwise inequality in time decay of the solution $u$ of an evolution PDE by means of the $L^1$-norm of the data $u_0$, namely ($t\neq 0$)
$$\|u(t,\cdot)\|_{L^\infty}\lesssim \frac{\|u_0\|_{L^1}}{|t|^r}$$
where (in general) the rate of decay $ r>0$ depends on the equation, the dimension and the setting.

\smallskip In the late 1990's, the seminal work of J. Ginibre and G. Velo\ccite{ginibrevelo} (see also M. Keel and T. Tao\ccite{keeltao}  for the endpoint) resorted  to dispersive estimates to allow for the derivation of Strichartz  estimates for a larger range of indices,  thanks to the  $TT^*$ argument recalled in Appendix\refer{sec_TT*}.
The research that has ensued in this subject  has been marked by a whole   series of dramatic results on dispersive and Strichartz estimates which are   the key to proving well-posedness results for nonlinear evolution equations.

\smallskip
We illustrate this idea with the $L^p-L^q$-Strichartz estimates for the Schr\"odinger equation. This result  contains, in particular, the first Strichartz estimates in~\eqref{eq:strlambda00}.
\begin{theorem}
\label{thm_StrichSchro}
For any  $u_0\in L^2(\R^{n})$, 
the solution $u$ to the Schr\"odinger equation in \eqref{eq_Schro} satisfies  the following family of Strichartz estimates:
\beq
\label {dispSTR}
\|u\|_{L^q(\R, L^p(\R^{n}))} \leq C(p,q)   \|u_0\|_{L^2(\R^{n})}\,;
\eeq
above $(p,q)$  satisfies the scaling admissibility condition
\beq
\label {admissibR}
\frac 2 q + \frac n p = \frac n 2  \with q \geq 2 \andf   (n,q,p) \neq (2,2, \infty) \, .
\eeq
\end{theorem}

\begin{proof}[Sketch of the proof of Theorem~\ref{thm_StrichSchro}]
An argument of complex interpolation between 
the dispersive estimate in~\eqref{dispSR0} 
with the conservation of the mass
$$ 
\|u(t,\cdot)\|_{L^2(\R^{n})} =\|u_0\|_{L^2(\R^{n})}\,,
$$ 
implies, for any $2 \leq p \leq \infty$,
\beq
\label {dispSR} \|u (t,\cdot) \|_{L^{p} (\R^{n})} \lesssim |t|^ {- n(\frac  {1} 2- \frac  {1} p)}  \|u_0\|_{L^{p'}(\R^{n})}\,.
\eeq
Invoking   the $TT^*$ argument,  this   gives rise  when $u_0\in L^2(\R^{n})$ to the following family of Strichartz estimates which contains the first Strichartz estimates in~\eqref{eq:strlambda00}:
\beq
\label {dispSTR}
\|u\|_{L^q(\R, L^p(\R^{n}))} \leq C(p,q)   \|u_0\|_{L^2(\R^{n})}\,;
\eeq
above $(p,q)$  satisfies the scaling admissibility condition in \eqref{admissibR}
 Indeed, if we denote  $ u (t, \cdot)= U  (t) u_0$, then  by definition of the $L^q_{t}(L^p_{x})$-norm, we have
\begin{eqnarray*}
\|U(t)u_0 \|_{L^q(\R, L^p(\R^{n}))}
&=&\sup_{\vf\in
\cB_{q,p}}\bigl| \int_{\R} (U(t)u_0|\vf (t))_{L^2(\R^{n})} \,dt\bigr|\, ,
\end{eqnarray*}
where $\cB_{q,p}=
\bigl\{\phi\in\cD(\R^{n})\,/\,\|\phi\|_{L^{q'}(\R, L^{p'}(\R^{n}))}\leq1\bigr\}$.  
By the definition of the adjoint operator, we thus have 
\begin{eqnarray*}
\|U(t)u_0\|_{L^q(\R, L^p(\R^{n}))}  
&\leq &
 \|u_0\|_{L^2(\R^n)} \sup_{\vf\in \cB_{q,p}} \Bigl\|\int_{\R} U^\star
(t)\vf(t, \cdot)\,dt\Bigr\|_{L^2(\R^n)} \, .
\end{eqnarray*}
However, we have by functional analysis and H\"older's inequality, 
\begin{eqnarray*}
\Big\|\int_{\R} U^\star(t) \vf(t, \cdot)\,dt \Big\|_{L^2(\R^n)} ^2  
 &   = &  \int_{\R^2}\bigl\langle
U(t-t')\vf(t', \cdot), \overline \vf (t, \cdot)\bigr\rangle \,dt'\,dt\\ &\leq& \int_{\R^2} \|U(t-t')\vf(t', \cdot)\|_{ L^{p}(\R^n)} \|\vf (t,\cdot)\|_{ L^{p'}(\R^n)}\,dt'\,dt\, ,
\end{eqnarray*}
which implies thanks to the  dispersive estimate\refeq{dispSR}
 $$ \Big\|\int_{\R} U^\star(t) \vf(t, \cdot)\,dt \Big\|_{L^2(\R^n)}^2 \leq \int_{\R^2} \frac 1 {|t-t'|^ { n(\frac  {1} 2- \frac  {1} p)}} \|\vf(t',\cdot)\|_{ L^{p'}(\R^n)} \|\vf (t,\cdot)\|_{ L^{p'}(\R^n)}\,dt'\,dt \,. $$
 This completes  the proof of  the result by Hardy-Littlewood-Sobolev inequality.  	
\end{proof}

  \medskip In fact, there is a plethora of Strichartz estimates expressed in  Lebesgue spaces as well as in Sobolev  spaces or more  generally in Besov spaces,  in homogeneous or inhomogeneous settings, whether for evolution equations with constant coefficients or   rough  variable coefficients.  For a brief introduction to this topic, we refer the reader to  the text\ccite{bahouri lp} and the references therein.  
  
    \smallskip In the particular case of  the Schr\"odinger equation in \eqref{eq_Schro}, combined with   a scaling argument,  the method of proof outlined above provides   a  gain of one half derivative with respect to the Sobolev embedding~$H^s(\R^n)\hookrightarrow L^\infty(\R^n)$, $s> n/2$, where $H^s(\R^n)$ stands for the inhomogeneous Sobolev space of regularity index $s$. As a result of this fact, nonlinear Schr\"odinger  equations can be solved for initial data less regular than what is required by the energy method. 
 
\smallskip To address the quasilinear framework, Strichartz   estimates for evolution equations  with variable (rough) coefficients have been also extensively studied: it is the combination of geometrical optics and harmonic analysis  that allows us to derive in that (more involved) framework   Strichartz estimates, with some loss compared to the constant coefficient case.   In  this context, the heart of the matter consists in constructing a 
a   microlocal parametrix, that is to say an approximation of $u_q$,  the part of the solution relating to frequencies of size~$2^q$, which  solves  a regular evolution equation (thanks to paradifferential calculus, see\ccite{BCD1}) but only on a small time interval whose size depends on the frequency. Families of dispersive estimates can then be established for    small time, leading to 
 microlocal Strichartz estimates by the $TT^*$ argument. 
 The local (in time) Strichartz estimate are then obtained by gluing the microlocal  estimates. We refer the interested reader to    H. Bahouri and J.-Y. Chemin\ccite{bch}, N. Burq,  P. G\'erard and N.  Tzvetkov \cite{bgt}, and  D. Tataru\ccite{Tataru}.

  \subsubsection{Evolution with lack of dispersion}
  Dispersion may not hold  or can be weak, such as for the wave equation on $\H_d$,  where   dispersive estimates have been established in\ccite{bgx} with an optimal rate of decay of order $|t|^{- \frac 1 2}$, regardless of the dimension~$d$. This is the case for instance, on compact Riemannian manifolds and on some bounded domains. 
  The Euclidean strategy described  in Section \ref{subsubsec_dispersion_strichartz}  breaks down, and then  obtaining Strichartz estimates is  considered a very difficult problem. 
  Other general approches can be used (with a possible loss of derivatives), such as the estimates  of resolvents  (see for instance\ccite{lebeau}),   the  microlocal analysis  as explained above\ccite{bch, bgt, Tataru}, or the use of adapted restriction results in the spirit of the pioneering work of R. Strichartz\ccite{strichartz}.

   \smallskip
  
    A model setting   for totally  non-dispersive evolution equations is  the Heisenberg group~$\H_d$ (see Example~\ref{ex_Hn}), 
   and more precisely the Schr\"odinger  equation for its canonical sub-Laplacian defined in~\eqref{eq_sublapHd}:
\begin{equation}
	\label{eq_SchroHeis}
	\quad \left\{
\begin{array}{c}
i\partial_t u -\D_{\H_d} u = f\\
u_{|t=0} = u_0\,,
\end{array}
\right.
\end{equation}
 Indeed,  
  in\ccite{bgx},   the authors  exhibited  examples   of Cauchy data $u_0$ for which the Heisenberg Schr\"odinger equation in \eqref{eq_SchroHeis}
  behaves as  a transport equation along the variable $s$ (called the vertical variable), in the sense that $$
u(t,Y,s) = u_0(Y,s+4td)\, .
$$  More generally,    it was emphasised in H. Bahouri, D. Barilari and I. Gallagher\ccite{BBG} that the Heisenberg Schr\"odinger equation in \eqref{eq_SchroHeis} can be interpreted   as  a superposition of transport equations $\pm T_{2\ell +d}$ along  the vertical direction, with velocity~$\pm4(2\ell +d)$,~$\ell \in \N$. 
  
\smallskip   Although the Schr\"odinger equation on the Heisenberg group in~\eqref{eq_SchroHeis} is considered totally non-dispersive,   the authors  in\ccite{BBG}  obtained a restriction result in the setting of~$\R \times \H_d$   for an adapted hyper-surface to the Schr\"odinger equation~\eqref{eq_SchroHeis},  that can be interpreted as the  parabola
 $S_{par}$ in the setting of $\R^n$; they also obtained appropriate analogue results   for the wave equation on $\H_d$. 
 They were  inspired by M\"uller's restriction result \cite{Muller}, and in particular took advantage of  Formula \eqref{defG}.
 Following Strichartz' approach\ccite{strichartz}, they  then established the following weak Strichartz estimate which shows the distinguished role of the vertical direction: 
\beq
\label{dispSTH}
\|u\|_{L^\infty_s L^{q}_t L^{p}_{Y}} \lesssim_{p, q}  \|u_0\|_{H^{ \frac {2d+2} 2 - \frac2q-\frac{2d}p}(\H_{d})}  \,, \eeq 
for all $(p,q)\in [2,\infty]^2$ such that $p\leq q$ and $ \frac2q+\frac{2d}p\leq \frac {2d+2} 2$. 
Here,  $H^{s}(\H_{d})$ denotes the Sobolev space on $\H_{d}$ of regularity index $s$; it  can be defined by several ways,   for instance via the functional calculus of $-\D_{\H_d}$ (see for instance\ccite{BBG, RLS} and the references therein).  In comparison with  \eqref{dispSTR}, one can easily check that the following weighted estimate holds  
\beq
\label{dispSTHbis}
\| \left \langle s \right \rangle^{- \alpha} u\|_{L^2(\R, L^2(\H_{d}))} \lesssim_\alpha  \|u_0\|_{L^2(\H_{d})}  \,, \eeq 
 for all $\alpha \geq \frac 1 2$, where~$\langle s  \rangle:=\sqrt {1+|s|^2}$. 
 
 \smallskip 
The restriction problem amounts to investigate an operator $\cP$ defined for instance spectrally as in Lemma \ref{lem_cP}. This operator may be further described  with the group Fourier transform (see Section\refer{subsubsec_L2T}) or/and via  its convolution kernel in terms of special functions.  To our knowledge, the first result in that  direction was achieved by D. M\"uller\ccite{Muller} on the Heisenberg group~$\bH_d$ in 1990.  M\"uller's result  created a surprise, since  while stressing that   there are no non-trivial solutions in the~$L^p$-spaces for $p>1$,  it provided a positive answer  in anisotropic Lebesgue spaces, see Section~\refer{sec_Liegroups}. 

    \smallskip Contrary to the Euclidean setting,  knowing restriction's theorems on~$\bH_d$ does not translate straightforwardly into results of the same type in  $\R\times \H_{d}$, which is of course not equal to~$\H_{d'}$, for some $d'$. Hence,  to get restriction results in the setting of~$\R \times \H_d$,  one   needs  to adapt the results pertaining to $\H_{d}$.

   \smallskip  Let us end this section by pointing out that, based on the restriction result of V. Casarino and P. Ciatti\ccite{CC},  D.  Barilari and  S.  Flynn\ccite{BF}  proved Strichartz estimates for the wave and Schr\"odinger  equations  for any $H$-type group, which are  examples of nilpotent Lie groups with two-step of nilpotency.

\subsubsection{Kato-smoothing effect} \label{smoothing}
 Variations  of the Tomas-Stein  estimate
\eqref{eq:estimesphere}, including refined restriction results or trace theorems, come also into play in the description  of  solutions of some   evolution equations. These estimates, known as Kato-smoothing properties,    have been  discovered independently by  A.-V. Faminskii and  S.-N. Kruzhkov\ccite{FK} and T. Kato\ccite{Kato1}  
 for the  Korteweg-de Vries  equation, then extended by P. Constantin and J.-C. Saut~\cite{Saut} to a large number of dispersive equations and systems on $\R^n$, see also \cite{BK,Kato}.  The Kato-smoothing effect for the Schr\"odinger equation on $\R^{n}$ writes  as follows:
 \begin{theorem}
 \label{thm_Kato}
For any  $u_0\in L^2(\R^{n})$, 
the solution $u$ to the Schr\"odinger equation in \eqref{eq_Schro} satisfies  the following Kato-smoothign estimate:
\begin{equation}\label{e1}
\|\left \langle x\right \rangle^{-1} \left \langle D_{x}\right \rangle^{\frac{1}{2}} u\|_{L^{2}(\mathbb{R}_{t}\times\mathbb{R}^{n})}\lesssim
\|u_{0}\|_{L^{2}(\mathbb{R}^{n})}, \end{equation}
   where~$\langle \cdot   \rangle:=\sqrt {1+|\cdot |^2}$.
 \end{theorem}  
   
\smallskip   The  estimate in \eqref{e1} tells us that the solution $u(t,\cdot)$ is, for almost $t$, locally~$1/2$ derivative smoother than the initial data~$u_{0}$. This regularization effect is different from the one displayed by Strichartz estimates, which as mentioned above can be only interpreted with respect to the Sobolev embedding: the estimate \eqref{e1} showcases the effective gain of the one half derivative.  Needless to say,  the  local character of Kato-smoothing effect is essential: indeed  due to the conservation of the $H^{s}$-norms along the flow of the Schr\"{o}dinger  equation, namely
\begin{equation}\label{us}
u_{0}\in H^{s}(\mathbb{R}^{n})\Leftrightarrow u(t,\cdot)=e^{-it\Delta}u_{0}\in H^{s}(\mathbb{R}^{n}),~\forall t\in\mathbb{R}\, ,
\end{equation}
a global smoothing effect is excluded in Sobolev spaces. 

\smallskip
The strategy of the proof of estimates  of type~(\ref{e1})  depends on the setting:   equations with constant   or variable coefficients,   on the whole space $\R^n$    or in exterior domains, etc.  For further details, one can consult  the monograph of L. Robbiano\ccite{RobLN}.  
Here we sketch the idea of  M. Ben-Artzi and S.  Klainerman  \cite{BK}.

\begin{proof}[Sketch of the proof of Theorem \ref{thm_Kato} following \cite{BK}]
The starting point is the following variant of Tomas-Stein restriction estimate 
\begin{equation}\label{3.1}
\|\mathcal{F}(f)|_{\wh\bS^{n-1}(r)}\|_{L^2 (\sigma_{\wh\bS^{n-1}} (r))}\lesssim \min{(r^{\frac{1}{2}},1)}\|\langle x  \rangle f\|_{L^2(\R^n)}\, ,
\end{equation}
where $\wh\bS^{n-1}(r)$ denotes the sphere centered at the origin and  of radius $r$ in  $\wh \bR^{n}$.    
 Its proof is mainly based on classical trace theorems in the framework of Sobolev spaces:  for $r \geq 1$, it is deduced   straightaway from the basic trace theorem,  while for  $r\leq 1$ it follows  from a combination of scaling arguments with the Hardy inequality and the application  of the following trace theorem on $\wh\bS^{n-1}$ 
$$ \begin{aligned}  & \|\mathcal{F}(f)|_{\wh\bS^{n-1}}\|_{L^2 (\sigma_{\wh\bS^{n-1}} )}\lesssim \|\mathcal{F}  (f)\|^2_{ H^1 (\hat B(0,1))}\lesssim \|\mathcal{F}  (f)\|^2_{ \dot H^1 (\hat \R^n)}\, \cdot \end{aligned}.$$
For further details, see the monograph of L. Robbiano\ccite{RobLN}.

With \eqref{3.1} at hand, the  idea  of  M. Ben-Artzi and S.  Klainerman  \cite{BK} consists first in reducing  the proof of the local smoothing property \eqref{e1}  by duality arguments to establish
\begin{equation}\label{goal_1}
\left|((\rm Id-\Delta)^{\frac{1}{4}}u|v)_{L^{2}(\mathbb{R}_{t}\times \mathbb{R}^{n})}\right|\le C\|u_{0}\|_{L^{2}(\mathbb{R}^{n})}\|\left\langle x\right\rangle v\|_{L^{2}(\mathbb{R}_{t}\times \mathbb{R}^{n})},
\end{equation}
for any test function $v(t,x)$.
Then, as  the spectral decomposition of the (self-adjoint extension of the) rescaled Laplace operator given by \eqref{eq_spectral}
 reads according to Fourier-Plancherel formula, for all~$f, g $ in~$\cS(\R^n)$,  
 \begin{align*}
(- (2\pi)^{-2}\Delta f | g)_{L^2(\bR^n)} &
=\int_{\wh \bR^{n}}
|\xi|^2 \cF f(\xi) \ \overline{\cF g}(\xi) \ d\xi
\\&=\int^\infty_{0}
r^2\biggl(\int_{\wh \bS^{n-1}} 
\cF f(r\omega) \ \overline{\cF g}(r\omega) 
d\sigma_{\wh\bS^{n-1}}  (\omega)\biggr)\, r^{n-1}dr\, .
\end{align*}
Therefore 
\begin{equation}\label{AA1}
(- (2\pi)^{-2}\Delta f | g)_{L^2(\bR^n)} = \int_{0}^{\infty} \mu (A(\mu) f | g)_{L^2(\bR^n)} d\mu\, ,\end{equation}
with  
\begin{equation}\label{AA}
(A(\mu)f|g)_{L^2(\bR^n)}=
\frac{1}{\sqrt{\mu}}( \mathcal{F}(f),  \mathcal{F}(g))_{L^2 (\sigma_{\wh\bS^{n-1}} ({\sqrt{\mu}}))}\,.
\end{equation}
Taking advantage of the  basic properties of the spectral measure of~$-\Delta$ on~$\mathbb{R}^{n}$,  one can then deduce that  the following bound holds
 \begin{align*}
|( (\rm Id- (2\pi)^{-2}\Delta)^{1/4} u|v)_{L^{2}(\mathbb{R}_{t}\times \mathbb{R}^{n})}| &\leq   \|u_{0}\|_{L^{2}(\mathbb{R}^{n})}\sqrt{\int_{0}^{\infty} \sqrt{1+\mu}\, (A(\mu) \tilde v (\mu, \cdot)| \tilde v (\mu, \cdot)) d\mu}\, .
\end{align*}
With the restriction estimate \eqref{3.1} and   Fourier-Plancherel formula with   respect to time, 
the conclusion follows.
\end{proof}

 \subsection{Some generalisations of Tomas-Stein estimates}\label{generalisations}
When functional inequalities are at hand, it is important to study their invariance by scaling, oscillations and translations,   etc,
as this often leads to their refinement. Functional inequalities were  one of Brezis'  predilection topics, where   he made deep and influential  contributions that have been at the origin of a large number of research projects, see for instance \cite{Bbook, Brezis, Bgal,  Blieb, Bmiro}. 

\subsubsection{Sobolev inequalities in Lebesgue spaces}
Among the most famous functional inequalities, we can mention the Sobolev inequalities in Lebesgue spaces: 
\begin{equation}
\label{inegsobolev} \dot H^s(\R^n)\hookrightarrow L^{p}(\R^n) \, ,\end{equation}
where $\dot H^s(\R^n)$ stands for the homogeneous Sobolev space of regularity index $s$,    $\ds 0\leq s<n/2$ and $\ds p=2n/(n-2s)$. Those inequalities are  invariant by translation and scaling, namely~$u_\lam:= u(\lam \cdot)$, but they are  not invariant 
by oscillations, that is, by multiplication by oscillating functions, namely by those of the type
$
u_\e(x) = e^{i\frac {(x|\omega)} \e} \vf(x),
$ 
where $\omega$ is a unit vector of $\R^n$, and $\vf$ is a function in~$\cS(\R^n)$. In 
\cite{gerardmeyeroru}, the authors P. G\'erard, Y. Meyer and F. Oru   refined\refeq{inegsobolev} as follows:
\begin{equation}
\label{inegrefined} 
\|u\|_{L^p(\R^n)}\lesssim \|u\|_{\dot B_{\infty, \infty}^{s-\frac n 2}(\R^n)}^{1-\frac 2 p}
\|u\|_{\dot H^s(\R^n)}^{\frac 2 p}\, ,\end{equation}
where $\dot B_{\infty, \infty}^{s-\frac n 2}(\R^n)$ denotes the homogeneous Besov space of regularity index $s-\frac n 2$. 
Hence, they have obtained a sharp inequality invariant under scaling, oscillations and translations. 
Recall that Besov spaces (which are in some sense generalisations of Sobolev spaces) can be defined in several ways (see for instance \cite{bahouri lp} for an overview of these spaces). In particular, for $\sigma >0$ and~$\psi\in \cS(\R^n)$ with compactly supported Fourier transform $\widehat \psi$, identically equal to $1$ near the origin, and  taking its values in $ [0, 1]$, one has (independently of $ \psi$)
$$\|u\|_{\dot B_{\infty, \infty}^{-\sigma}(\R^n)}:= \sup_{A>0} A^{n-\sigma} \|\psi (A \cdot ) \star u\|_{L^\infty(\R^n)}=\sup_{A>0} A^{-\sigma} \|\cF^{-1} (\widehat \psi ((A^{-1}  \cdot ) \widehat u)\|_{L^\infty(\R^n)}\,.$$
The refined estimate \eqref{inegrefined}  is one of the key arguments in \cite{Ge2},  where P. G\'erard
gave a characterisation of the defect of compactness of the critical Sobolev embedding \eqref{inegsobolev} by means 
of the profile decompositions. This characterisation can be formulated in the following terms: a bounded sequence $(u_n)_{n\geq 0}$  in $\dot H^{s}(\R^n)$
can be decomposed (up to a subsequence extraction) according to
\beq
\label{profdec}
u_n=\sum_{l=1}^L h_{l,n}^{s-n/p}\phi^l(\frac {\cdot - x_{l,n}}{h_{l,n}}) + r_{n,L}\,, \, \, \lim_{L\to +\infty}(\limsup_{n\to +\infty}\|r_{n,L}\|_{L^{p}})=0\,,
\eeq
where $(\phi^l)_{l\geq 1}$ is a family of functions
in $\dot H^{s}(\R^n)$ and where the cores $(x_{l,n})_{n\in \N}$ satisfy,
 for all~$k\neq l$,
$$
|\log (h_{l,n}/h_{k,n})| \to +\infty
\;\;{\rm or}
\;\;
|x_{l,n}-x_{k,n}|/h_{l,n}\to +\infty,
\;\;{\rm as}\;\; n\to +\infty.
$$
In short, the decomposition \eqref{profdec} shows that translational and scaling invariance
are the only features responsible for the defect of compactness of the Sobolev embedding \eqref{inegsobolev}.

\smallskip

Nonlinear analysis   progressed substantially
 in the last decades due to profile decomposition techniques that  originated  from the work of H. Br\'ezis and J.-M. Coron\ccite{BC} motivated by geometric problems. As can be seen from the extensive literature on the subject, these techniques are now essential  in PDEs in the contexts of both elliptic and evolutionary equations. A vast  literature  on the subject has been growing: among others, let us mention  
 H. Bahouri and G. Perelman\ccite{BaPe},  M. Del Pino, F.  Mahmoudi and M.  Musso\ccite{delPino},  J. Jendrej and  A. Lawrie\ccite{JL}, E. Lenzmann and  M. Lewin\ccite{LL},  F. Merle and L. Vega\ccite{merlevega},    
C.-E. Kenig and  F. Merle\ccite{km}, and the references therein. 
  Let us also emphasise that from the profile decompositions\refeq{profdec},  one can  recover that the best Sobolev constant in \eqref{inegsobolev} is reached,  as well as the structure of minimising sequences (see the founding paper of P.-L. Lions\ccite{Lions}).

\medskip  The question of optimal  constants and the existence of extremisers  in functional inequalities such as Sobolev embedding or   Gagliardo-Nirenberg inequalities   is a long standing problem that has been the subject of many papers for the sake of geometric problems, or physical studies in many settings, but remains a topical issue. This kind of analysis can be carried out by variational analysis,  optimal transport methods (see for instance\ccite{CNV}) or concentration compactness arguments as mentioned above.

\subsubsection{Case of the Heisenberg group}
 In\ccite{Benameur}, J. Ben Ameur obtained, in the framework of the Heisenberg group, an appropriate analogue profile decomposition to\refeq{profdec}. J. Ben Ameur's result characterises  the defect of compactness of  the  Sobolev embedding: 
\begin{equation}
\label{inegsobolevH} \dot H^s(\bH_d)\hookrightarrow L^{p}(\bH_d) \, ,\end{equation} where $\dot H^s(\bH_d)$ denotes   the homogeneous Sobolev space on $\bH_d$ of regularity index $s$, with~$\ds 0\leq s<Q/2$ and $\ds p=2Q/(Q-2s)$, and where  $Q:=2d+2$ stands for the homogeneous dimension of $\bH_d$ given by  the anisotropic dilations\refeq{dil}. Based on this decomposition,   L. Gassot\ccite{gassot} showed that the best Sobolev constant is attained in the setting  of~$\bH_d$, and also  constructed families of  traveling waves for the cubic Schr\"odinger equation on  the Heisenberg group.

\subsubsection{Refined Tomas-Stein estimates} Concerning the Tomas-Stein inequality, similar studies have been conducted. On the one hand,  the literature includes   several  refinements of the functional inequality \eqref{eq:estimesphere}: among others, one can cite  R.-L. Frank, E. Lieb and J. Sabin\ccite{Sabin},    R. Killip and M. Visan\ccite{killipvisan},  A. Moyua, A. Vargas and L. Vega\ccite{MVV},   D. Oliveira e Silva\ccite{OS}, S. Shao  \cite{Shao}, T. Tao\ccite{Tao2}, etc.  
In particular, one has the following refined estimate to be compared with \eqref{inegrefined}:  there exists a positive constant $C$,  such that for all~$g \in L^{2}(\wh \bS^{n-1}, d\sigma_{\wh \bS^{n-1}})$,
\beq
\label{refined} \|\cF^{-1} (gd\sigma_{\wh \bS^{n-1}})\|_{L^{\frac {2n+2}{n-1}} (\R^n)}\lesssim  \big(\sup_{Q \in D} |Q|^ {- \frac 1 2} \|\cF^{-1} (1_{Q} g d\sigma_{\wh \bS^{n-1}})\|_{L^\infty (\R^n)}\big)^\theta \|g\|^{1-\theta}_{L^{2}(d\sigma_{\wh \bS^{n-1}})}, \eeq
for  some~$0<\theta<1$, where $D$ denotes a family of dyadic cubes covering $\wh \bS^{n-1}$.  Note that\refeq{refined} implies the standard Tomas-Stein  inequality: \beq
\label{refinedbis} \|\cF^{-1} (gd\sigma_{\wh \bS^{n-1}})\|_{L^{\frac {2n+2}{n-1}} (\R^n)}\lesssim    \|g\|_{L^{2}(d\sigma_{\wh \bS^{n-1}})}. \eeq Indeed, for any $Q \in D$, we have by Riemann-Lebesgue theorem
$$\begin{aligned} |Q|^ {- \frac 1 2} \|\cF^{-1} (1_{Q} g d\sigma_{\wh \bS^{n-1}})\|_{L^\infty (\R^n)} &\lesssim |Q|^ {- \frac 1 2} \|1_{Q} g\|_{L^1(d\sigma_{\wh \bS^{n-1}})} \\ &\lesssim |Q|^ {- \frac 1 2} \|1_{Q} \|_{L^{2}(d\sigma_{\wh \bS^{n-1}})} \|g\|_{L^{2}(d\sigma_{\wh \bS^{n-1}})}\\ &\lesssim \|g\|_{L^{2}(d\sigma_{\wh \bS^{n-1}})}\, .\end{aligned}$$
The proof of the various refined Tomas-Stein estimates  are different from that of\refeq{inegrefined}.  While  the proof of \refeq{inegrefined} results from a combination of   Cavalieri principle with the classical method of decomposition  of  functions into low
and high frequencies (see for instance Theorem 1.43 in\ccite{BCD1}),  that of the refined estimates of  Tomas-Stein inequality relies rather on   bilinear restriction results. 
In particular,  the proof of \refeq{refined} relies   on a deep bilinear restriction
theorem of Tao\ccite{Tao2}.  Recall that such bilinear estimates  lead to   a better gain of regularity  for the product of two solutions of   the Schr\"odinger equation in~\eqref{eq_Schro}
compared with the product laws and the gain of one half derivative (with respect to the Sobolev embedding)   with purely Strichartz methods.

\smallskip On the other hand, along the same lines as in the proof of the profile decomposition\refeq{profdec} where \eqref{inegrefined} plays a key role,   the refined estimate \eqref{refined} involves  in  a sort of "Knapp" profile   decomposition, that  have been introduced  by  M. Christ and S. Shao\ccite{CS}. Roughly speaking, the "Knapp" algorithm    decomposition implemented in\ccite{CS} is built on the fact that if a bounded sequence~$(f_k)_{\in \N}$ of~$L^{2}(\wh \bS^{n-1})$ do not concentrate around   the north pole (up the invariance of \eqref{refinedbis}), then (see Lemma 5.2 in\ccite{Sabin}) $$\sup_{Q \in D} |Q|^ {- \frac 1 2} \|\cF^{-1} (1_{Q} f_k d\sigma_{\wh \bS^{n-1}})\|_{L^\infty (\R^n)} \stackrel{k\to\infty}\longrightarrow 0\, .$$    
Let us also underline   that similarly to the case of  the Sobolev embedding \eqref{inegsobolev},  the refinements of Tomas-Stein inequalities as well as  the "Knapp" profile   decompositions come to play in the study of best constants and extremisers involved in this setting, which is not completely settled in
its general form and remains a topical issue. Indeed, works addressing the existence of extremisers for Tomas-Stein inequalities tend to be a tour de force, using a variety of tools and technics,  such as   the  refined estimates,   bilinear estimates, and the concentration compactness arguments mentioned above. 
For an overview about this topic, see for instance the recent paper Foschi-Oliveira e Silva\ccite{Foschi}.

\subsubsection{Other frameworks}  We have seen in Section \ref{Muller} that in the  particular case of  the Heisenberg group~$\mathbb{H}^{d}$, an estimate of type\refeq{eq:estimesphere} can only hold for $p=1$, which is the trivial index given by Riemann-Lebesgue theorem!  However,   M\"uller\ccite{Muller}   established an appropriate  restriction estimate  involved anisotropic Lebesgue spaces.
 This prompts us to generalise the question of Stein in the following way: given a Lie group $G$, for which  norm~$N_G$ on $G$, an inequality of the form
\begin{equation}
\label{eq:estime0}
 \|\cF_G( f)|_{\wh S_G}\|_{L^2 (d \sigma_{\wh S_G})}\leq C N_G(f)\, ,
\end{equation}
holds for all $f$ in  the Schwartz space $\cS(G)$, where  $\cF_G (f)$ denotes the Fourier transform of~$f$ on the group $G$, and $\wh S_G$ stands for the dual unit sphere. 
 The definition of the dual unit sphere $\wh S_G$ and its operator-valued measure $\sigma_{\wh S_G}$ is open to interpretation. 
In the case of the Heisenberg groups, the dual unit sphere and its operator-valued measure are defined in spectral terms using the relation of the spectral decomposition of the canonical sub-Laplacian and the group-Fourier transform  as in Section~\ref{Muller}.
 This is more generally the case on  Carnot groups  (see Section \ref{subsubsec_technicalDef} for a definition) as they are equipped with a canonical sub-Laplacian.

\medskip 

Let us end this section  by stressing that numerous works related to restriction problems and its applications  have been carried out in the framework of hyperbolic geometry, especially the hyperpolic space (one model of which is the hyperbola as in Section \ref{subsubsec_exGhomdom}). Regarding restriction theorems, one can consult the recent work of \cite{BDV},  and the references therein.       For cluster estimates, one can see  the  latter paper of  J.-P. Anker, P. Germain and T. L\'eger\ccite{AGL}, while  for Strichartz estimates,   one can look for instance to the studies by   V. Banica\ccite{Banica} as well as that by  J.-P. Anker   and V. Pierfelice\ccite{AP}   concerning the Schr\"odinger equation. Note finally that the issue of extremals for Gagliardo-Niremberg inequalities in hyperbolic geometry has been investigated by   M.  Mukherjee in\ccite{MM}, whereas, as far as we know,  the study  of extremals of restriction's estimates in hyperbolic geometry is still an open question.


\subsection{Discrete framework}\label{des-case}
  In the discrete case, we have to deal with exponential sums, such as for instance
\begin{equation}\label{eq:gendis} f_A (x)= \sum_{a \in A} c_a e^{2i\pi a \cdot x} \, ,\end{equation} 
 with $A$ a subset of $\Z^n$.  Evaluating $f_A$  in $L^p ([0, 1]^n)$ is   a challenging issue: the involved 
terms have various phases in the complex plane, and it is difficult to tell what happens when we add
them all up.

\subsubsection{The first $L^2$-restriction theorem on the flat Torus} \label{zyg}  To our knowledge, the first restriction result  on the flat Torus dates back to the   seventies with the result of A. Zygmund\ccite{zygmund} concerning the two dimensional case. 
\begin{theorem} [\cite{zygmund}] 
 \label{thm_1restperiodic}
There exists a positive constant $C$ such that,  for any regular function $f$ on $Q:=[0, 1]^2$ and any $r>0$, we have 
 \begin{equation}\label{n=2Torus}   \big(\sum_{|\mu| =r} |c_\mu|^2 \big)^{\frac 1 2} \leq C  \|f\|_{L^{\frac 4 3}(Q)}\, ,\end{equation}
where, for $\mu= (m, n) \in \Z^2$,  $c_\mu$ denotes the Fourier coefficient of $f$ given by  $$c_\mu:=\int_Q e^{-2i \pi \mu\cdot x} f(x) dx\, .$$ 
 \end{theorem} \begin{proof}[Sketch of the proof of Theorem \ref{thm_1restperiodic}] Assume that $\widehat S_r$  the subset  of points $\mu=(m,n)$ in the lattice $\Z^2$ satisfying $|\mu| =r$, namely $m^2 +n^2 =r^2$, is non empty, and write 
\begin{eqnarray*}\big(\sum_{|\mu| =r} |c_\mu|^2 \big)^{\frac 1 2} &= &\sum_{|\mu| =r} c_\mu \gamma_\mu \\ &=& \int_Q f(x) \big[\sum_{|\mu| =r} \gamma_\mu e^{-2i \pi \mu\cdot x}\big]  dx \, ,\end{eqnarray*}
 with   $\ds \gamma_\mu:= \frac {\overline c_\mu}{\|c_\mu\|_{\ell^{2}(\widehat S_r)}}\cdot$ 
 
 \smallskip Applying H\"older's inequality, and then Parseval's formula, we deduce that 
 \begin{eqnarray*}\big(\sum_{|\mu| =r} |c_\mu|^2 \big)^{\frac 1 2} &\leq  & \|f\|_{L^{\frac 4 3}(Q)} \|\sum_{|\mu| =r} \gamma_\mu e^{-2i \pi \mu\cdot x}\|_{L^{4}(Q)}  \\ &=& \|f\|_{L^{\frac 4 3}(Q)} \|\Gamma_\rho \|^2_{\ell^{2}(\Lambda)}  \, ,\end{eqnarray*} 
 where $\Lambda$ denotes the lattice of points $\rho \in \Z^2$ so that $\rho=\mu- \nu$, with $|\mu| = |\nu|  = r$  and  \begin{equation}
\label{eq:rho}\Gamma_\rho:=\sum_{ \mu- \nu=\rho  \atop  |\mu| = |\nu|  = r}   \gamma_\mu {\overline \gamma_\nu} \, .\end{equation}
By definition  $\Gamma_0=\|\gamma_\mu\|^2_{\ell^{2}(\widehat S_r)}=1$ while, for $\rho  \neq 0$ (belonging to the lattice $\Lambda$) $\Gamma_\rho$ includes at most two terms, which easily (according to the fact that  $\|\gamma_\mu\|_{\ell^{2}(\widehat S_r)}=1$) allows us  to conclude the proof of the result. \end{proof}
\smallbreak

The proof of Zygmund's theorem highlights the connexion between the restriction problem in the discrete case with Number Theory.  

\smallskip

  
\medskip

\subsubsection{Schr\"odinger equations on flat tori}
The Schr\"odinger equation makes sense on any Riemannian manifold, and in  particular on  flat tori. However,  understanding the analogue of Strichartz estimates \eqref{eq:strlambda00} on compact  manifolds $ (M, g)$ is
extremely difficult. The first examples where sharp Strichartz estimates were proved concern the case of $S^1$ by J. Bourgain\ccite{bourgain3}  and $S^3$ by N. Burq, P. G\'erard and N. Tzvetkov\ccite{bgt}.  Let us discuss briefly the main arguments for the above theorems. 

\smallskip In the case of the flat Torus~$\T^{n}=\R^{n}/\Z^{n}$, using the Fourier series theory, one can write any solution of the Schr\"odinger equation under  the form:
\begin{equation}\label{Torus1} u(t,x)= \sum_{k\in \Z^n} a_k e^{2i\pi (k\cdot x+ |k|^2t)} \, ,\end{equation}
where $(a_k)_{k \in \Z^n} $ denote  the Fourier coefficients   of the Cauchy data  $u_0 \in L^2(\T^{n})$, and then of course~$\|u_{0}\|_{L^{2}(\T^{n})}=\|a_k\|_{\ell^2(\Z^{n})}$.  

\smallskip In the one dimensional  case,   J. Bourgain in\ccite{bourgain3}  established the following sharp Strichartz estimate 
\begin{equation}\label{Torus-st} \|\sum_{n \in Z} a_n e^{2i\pi (n  x+ n^2t)}\|_{L^4(\T^{2})}  \lesssim  \|a_n\|_{\ell^2(\Z)} \, .\end{equation}
Its    proof is   in the same vein as the proof of the estimate \eqref{n=2Torus}  by A. Zygmund\ccite{zygmund} 
outlined in Section \ref{zyg}:  setting~$f (t,x) =\sum_{k \in Z} a_k e^{2i\pi (k\cdot x+ k^2t)}$ and observing that $\|f\|^2_{L^4(\T^{2})}=\|f \overline f\|_{L^2(\T^{2})}$, where
$$  (f \overline f) (t, x)= \sum_{k \in Z}  |a_k|^2+ \sum_{k_1 \neq k_2} a_{k_1}  \overline a_{k_2}e^{2i\pi ((k_1-k_2)  x+ (k_1^2-k_2^2)t)}\, ,$$
one can easily infer that 
$$\|f\|^2_{L^4(\T^{2})} \leq 2  \|f\|_{L^2(\T^{2})}\, ,$$
which implies \eqref{Torus-st} by Parseval identity. Here again as for \eqref{n=2Torus},   we are dealing with an easy example of arithmetic structure, and  this allows us to end up easily with the result.

\medskip However, higher dimensions are more challenging: actually  the sum\refeq{Torus1} is large near rational points of the form~$\ds (\frac {p_1}q, \cdots, \frac {p_d}q, \frac {p_t}q)$, and this makes the study of such sums tricky.  Optimal results in this setting have been obtained later on  thanks to the decoupling method introduced by T. Wolff \cite{Wolff}, and developed later on by several authors, in particular by J. Bourgain and C. Demeter\ccite{{bourgain-demeter}}. 

\medskip   On compact Riemannian manifolds $(M, g)$,  a
natural expansion of square integrable functions is provided by  the spectral decomposition
of the associated Laplace-Beltrami operator. 
However, the  knowledge of the spectrum and of
the eigenfunctions of such  operators on arbitrary compact  manifolds $(M, g)$  is too poor
to be able  to adapt the  method applied for   flat Tori. The strategy adopted by the authors in\ccite{bgt} is rather  
 the one described in Section\refer{Strichartz}, which relies on microlocal approximations.    
 
 \medskip Decoupling is a recent method in Fourier analysis that is helpful for estimating $\|f\|_{L^p}$, for $p\neq 2$,  in terms of informations about the Fourier transform $\mathcal{F}( f)$.  One impressive  application of this method involves Strichartz estimates for the Schr\"odinger equation on the flat Torus~$\T^{n}=\R^{n}/\Z^{n}$.  When $u$ has frequencies at most $N$, one has the following result
due to Bourgain-Demeter:
\begin{theorem}[\cite{{bourgain-demeter}}] \label{t:tsbasest}
{\sl If the  Fourier coefficients  $(a_k)_{k \in \Z^n} $  are supported in 
$$
Q_N=\{(k_1,\cdots,k_n)\in \Z^n 
\colon |k_j| \leq N,  \, \,  j=1,\ldots, n \},
$$ 
then the solution of the Schr\"odinger equation on~$\T^{n}$ given by 
$$u(t,x)= \sum_{n\in Q_N} a_n e^{2i\pi (n\cdot x+ |n|^2t)}$$
satisfies
\beq
\label{dec}
\|u\|_{L^{\frac {2(d+1)}  d }(\T^{n} \times [0, 1])} \lesssim_\varepsilon N^\varepsilon \|a_n\|_{\ell^2}\,.
\eeq  
}
\end{theorem}

\medbreak

 Decoupling is a recent development in Fourier analysis whose basis were developed within the framework of the theory of restriction. In abstract terms, if $\mathcal{F}( f)$ is supported in $\Omega \subset \R^n$, then we decompose $$f=\sum_\theta f_\theta,$$ where $\Omega =\cup \theta$ is a disjoint union of subsets $\theta$, and  $\mathcal{F}( f_\theta)$ is supported in $\theta$, namely 
$$f_\theta(x)= \int_\theta \mathcal{F}( f) (\omega) e^{2i \pi \omega x} d\omega.$$
For each $p$, we define the decoupling constant $D_p:=D_p(\Omega =\cup \theta)$ to be the smallest constant so that, for all $f$ with $\mathcal{F}( f)$  supported in $\Omega$,
$$ \|f\|^2_{L^{p}(\R^{n})} \leq D^2_p \sum_\theta   \|f_\theta \|^2_{L^{p}(\R^{n})} .$$ 
For most applications, we are interested by sets $\theta$ which are partition of a thin neighborhood of a curved manifold, for instance a truncated parabola in Theorem \ref{t:tsbasest}. For an introduction to this method, one can consult the survey of   L. Guth \cite{Guth} and the references therein.

\subsection{Number theory connection}\label{NT}
Concerning  connections  to Number  Theory, Additive Combinatorics and PDEs in discrete setting, one can consult\ccite{DemeterICM, DemeterBourgain, Guth} for an introduction to the decoupling methods whose foundations have grown inside the framework of  restriction theory. 
\subsubsection{Vinogradov's conjecture}\label{Vinogradov} Let us illustrate the interplay between restriction's problem  and  Number  Theory with Vinogradov's conjecture which concerns diophantine systems: 

\smallskip
{\it For fixed positive integers $s, k, N$, let $J_{s, k}(N)$ be the number of integer solutions 
$$ k^j_1+ \cdots + k^j_s= k^j_{s+1}+ \cdots + k^j_{2s}\, ,    $$ for all $1 \leq j \leq k $ with $(k_1,\cdots, k_{2s}) \in  \{1, \cdots, N\}^{2s}$.}  
\smallskip

The Vinogradov's question,  which dates back to the 1930's, focuses on the asymptotics of $J_{s, k}(N)$, for $N$  large enough. 

\smallskip Using analytic number theory, Vinogradov have proved good estimates for $J_{s, k}(N)$, for some values of $k$ and $s$. Later on in 2015, based on Fourier restriction ideas, and in particular  on the decoupling method, J. Bourgain, C. Demeter and L. Guth  \cite{BDG} proved the following  asymptotic sharp bounds (see also the papers of T. Wooley\ccite{Woo1, Woo}  using  the efficient congruencing): for all $\epsilon>0$, \begin{equation}\label{eq:Vin}
J_{s, k}(N) \lesssim_{s, k, \epsilon}      N^\epsilon (N^s + N^{2s - \frac {k (k+1)} 2 })\, ,
\end{equation}  for $N$ sufficiently large. 

\smallskip In particular, one has, in the case when $\ds 1 \leq s \leq \frac {k (k+1)} 2$, 
\begin{equation}\label{eq:Vinbis} J_{s, k}(N) \lesssim_{s, k, \epsilon}      N^{s+\epsilon}  \, .
\end{equation} Observing that there are $N^s$ trivial solutions, those with $(k_1,\cdots, k_{s})=(k_{s+1},\cdots, k_{2s})$, the bound \eqref{eq:Vinbis} shows that when $\ds 1 \leq s \leq \frac {k (k+1)} 2$, there are not too many non trivial solutions.  

\smallskip At first glance, this conjecture  linked to Number Theory seems to be very far from the restriction's problem! This is not the case, since it turns out   that the number $J_{s, k}(N)$ admits the following analytic representation:
\begin{equation}\label{eq:Vindec} J_{s, k}(N) = \int_{[0, 1]^k} |\sum^N_{j=1} e^{2i\pi (jx_1+j^2x_2+\cdots+j^k x_k)}|^{2s} dx_1\cdots dx_k\, ,
\end{equation} which by straightforward computations follows from the obvious fact that 
$$ \int_{[0, 1]} e^{2i\pi m x} dx = \left\{\begin{array} {ccl}
1 &\mbox{if} & m=0\\
0 &\mbox{if} & m \neq 0 \, .
\end{array}\right.
  $$

This formulation  can be seen as a discrete analogue of Tomas-Stein's estimate \eqref{dual2}, where the measures $gd\sigma_{\wh S}$ are sums of exponentials with discrete frequencies located on manifolds such as the parabola when $k=2$ or the momentum curve, for $k\geq 3$, $$\{(t, t^2,\cdots, t^k)\colon\,\, t \in [0, 1]\} \, .$$

\smallskip The  example of Vinogradov's conjecture clearly illustrates the connection between Number and Restriction Theory. We will not comment further on this fact  in this text, but we refer the interested reader  to the survey of   L. Guth \cite{Guth} and the references therein.

\subsubsection{Additive combinatorics}  Nilmanifolds appear in the study of arithmetic progressions and additive combinatorics, especially in works by B. Green and T. Tao, see e.g.\ccite{GreenTao2008,GreenTao2012}.
For instance, the Green-Tao theorem states that the primes contain arbitrarily long arithmetic progressions, and its proof relies on understanding the distribution of primes in `nilsequences,' which are sequences arising from flows on nilmanifolds.

\subsection{Open problems}
As is traditional in Brezis' papers, we conclude this text with   a few open questions related to restriction Stein's problem.

\medskip

 From the above cited research a few inquiries arise: 
 \begin{enumerate}  \item The first one   concerns  hyper-surfaces $\wh S$ of $\wh \bR^n$  whose Gaussian curvature vanishes  on submanifolds of $\wh S$ with codimension $k>1$ (and do not anywhere else).  This is for instance the case of the   revolution Torus in $\wh \bR^3$ which has for  equation  
 $$
 \big(\sqrt{x^2+y^2}-R\big)^2+z^2=r^2,  \,\, 0< r< R \, ,
 $$ whose  Gaussian curvature vanishes on the two circles  at height $z=\pm r$, and is positive anywhere else. One   wonders what should be the corresponding optimal estimate for Stein's problem. 

 \smallskip
 
\item   In Section\refer{Strichartz}, we raised the issue of quasilinear equations which requires more elaborate techniques than the constant coefficient case. In this setting, the heart of the matter consists to investigate approximate solutions under the form: 
$$ \int e^{i \Phi (t, x, \xi)} \sigma (t, x, \xi) d \xi \,,$$ for some phase $\Phi$ and amplitude $\sigma$.  In the articles\ccite{bch, bgt, Tataru} cited above, the involved oscillating integrals are investigated combining microlocal analysis and  stationary phase theorem. It would be helpful to establish restriction theorems, under some hypothesis on the phase $\Phi$, for instance as some  perturbation of the phase involved in\refeq{eq:str0}. Note that   L. Guth, J. Hickman and M. Iliopoulou provide in\ccite{GuthHL} a  restriction  result  for a  parametrix's class.  

 \smallskip 
 
 \item Tomas-Stein estimates   have appropriate analogues   in the framework of hyperbolic geometry.  However, to our knowledge,   refined versions in the spirit of\refeq{refined} have not been tackled. Although  the roadmap of Section\refer{generalisations} is not fully resolved in the Euclidean setting, it would be interesting to refine restriction estimates in the hyperbolic geometry and to investigate adapted concentration compactness tools, with the aim of investigating the involved optimal constants and extremals. 
\smallskip

 \item It is anticipated that cluster estimates for sub-Laplacians will not behave like their Riemannian counterparts  in Theorem~\ref{thm_cluster}, especially when sharp. 
This intuition is based on the difference in propagation properties and Strichartz estimates for nilpotent Lie groups, especially the Heisenberg groups in sections\refer{Restriction-Heisenberg} and\refer{App}.

A concrete  open problem is to find cluster estimates (even non-sharp)  on Heisenberg nilmanifolds for the canonical sub-Laplacians. 
The spectral theory of these sub-Laplacians are completely explicit\ccite{DeningerSinghof}, so cluster estimates in these settings may be viewed as related to analytic number theory. 
Other approaches (e.g. via Harmonic Analysis or Spectral Theory)
may  also be possible. 
\end{enumerate}

\appendix 

 \section{The $TT^*$ argument}
 \label{sec_TT*}
 
 Here, we recall the $TT^*$ argument as presented in 
 \cite[Lemma 2.2]{ginibrevelo}.
 
 \smallskip
  
 If $Y$ and $Z$ are two vector spaces (without any assumption of norms or topology), we denote the space of linear maps $Y\to Z$ by $\sL_a(Y,Z)$.
 
 Let $\cH$ be a Hilbert space.
Let  $X$ a Banach space; the Banach space dual
to $X$  is denoted by  $X^*$.
Let $\cD$ be a vector space densely contained in $X$.

If $A\in \sL_a(\cD, \cH)$, we define its algebraic adjoint
 $A^* \in \sL(\cH, \cD_a^*)$ in the following way: the space  
$\cD_a^* =\sL(\cD,\bC)$ is  the algebraic dual of $\cD$ equipped with the  algebraic pairing  $\langle \cdot , \cdot \rangle_\cD$, and 
 the operator $A^*$ is defined via
$$
\forall f\in \cD,\quad\forall v\in \cH,\qquad  
\langle A^* v, f\rangle_\cD= \langle v, Af\rangle_{\cH},
$$
where the scalar product 
 $\langle \cdot , \cdot \rangle_\cH$ on $\cH$ is 
conjugate linear in the first argument.

\begin{theorem}
	Under the hypotheses above,  the following three conditions
are equivalent:
\begin{enumerate}
	\item There exists $a\in [0,\infty)$ such that 
$$
\forall f\in \cD,
\qquad
\|Af\|_\cH \leq a \| f\|_X
$$
\item The range of $A^*$ is included in $X^*$ and  there exists $a\in [0,\infty)$ such that 
$$
\forall v\in \cH,
\qquad 
\|A^*v\|_{\bX^*}\leq a \| v\|_\cH
$$
\item The range of $A^*A$ is included in $X^*$ and  there exists $a\in [0,\infty)$ such that 
$$
\forall f\in \cD,
\qquad
\|A^*Af\|_{\bX^*}\leq a^2 \| f\|_X.
$$
\end{enumerate}
If one of (all) those conditions is (are) satisfied, the operators $A$ and
$A^* A$ extend by continuity to bounded operators  $X \to \cH$ and  $X\to X^*$ respectively.
In this case, we may take $a\in [0,+\infty)$ to be the same constant  in all three
parts with
$$
a\geq \|A\|_{\sL(X,\cH)} =  \|A^*\|_{\sL(\cH,X^*)} = \sqrt{  \|A^*A\|_{\sL(\cH,X^*)}}.
$$
\end{theorem}

\end{document}